\documentclass[10pt]{amsart}
\usepackage{amsfonts}
\usepackage{mathrsfs}
\usepackage{amscd}
\usepackage{amsmath}
\usepackage{amssymb}
\usepackage{latexsym}
\usepackage{lscape}
\usepackage{xypic}
\usepackage{comment}
\usepackage{amscd}
\usepackage{wasysym}  
\usepackage{tikz} 
\usetikzlibrary{matrix,arrows}
\usepackage{appendix}

\newtheorem{thm}{Theorem}[section]
\newtheorem{Thm}[thm]{Theorem}
\newtheorem{prop}[thm]{Proposition}
\newtheorem{lem}[thm]{Lemma}

\newtheorem{lem-def}[thm]{Lemma-Definition}
\newtheorem{cor}[thm]{Corollary}

\theoremstyle{definition}
\newtheorem*{ack}{Acknowledgement}
\newtheorem{ex}[thm]{Example}

\newtheorem{rmk}[thm]{Remark}
\newtheorem{dfn}[thm]{Definition}

\numberwithin{equation}{section}

\newcommand{\nc}{\newcommand}

\nc{\on}{\operatorname}
\nc{\fraka}{{\mathfrak a}}
\nc{\frakb}{{\mathfrak b}}
\nc{\frakc}{{\mathfrak c}}
\nc{\frakd}{{\mathfrak d}}
\nc{\frake}{{\mathfrak e}}
\nc{\frakf}{{\mathfrak f}}
\nc{\frakg}{{\mathfrak g}}
\nc{\frakh}{{\mathfrak h}}
\nc{\fraki}{{\mathfrak i}}
\nc{\frakj}{{\mathfrak j}}
\nc{\frakk}{{\mathfrak k}}
\nc{\frakl}{{\mathfrak l}}
\nc{\frakm}{{\mathfrak m}}
\nc{\frakn}{{\mathfrak n}}
\nc{\frako}{{\mathfrak o}}
\nc{\frakp}{{\mathfrak p}}
\nc{\frakq}{{\mathfrak q}}
\nc{\frakr}{{\mathfrak r}}
\nc{\fraks}{{\mathfrak s}}
\nc{\frakt}{{\mathfrak t}}
\nc{\fraku}{{\mathfrak u}}
\nc{\frakv}{{\mathfrak v}}
\nc{\frakw}{{\mathfrak w}}
\nc{\frakx}{{\mathfrak x}}
\nc{\fraky}{{\mathfrak y}}
\nc{\frakz}{{\mathfrak z}}
\nc{\frakA}{{\mathfrak A}}
\nc{\frakB}{{\mathfrak B}}
\nc{\frakC}{{\mathfrak C}}
\nc{\frakD}{{\mathfrak D}}
\nc{\frakE}{{\mathfrak E}}
\nc{\frakF}{{\mathfrak F}}
\nc{\frakG}{{\mathfrak G}}
\nc{\frakH}{{\mathfrak H}}
\nc{\frakI}{{\mathfrak I}}
\nc{\frakJ}{{\mathfrak J}}
\nc{\frakK}{{\mathfrak K}}
\nc{\frakL}{{\mathfrak L}}
\nc{\frakM}{{\mathfrak M}}
\nc{\frakN}{{\mathfrak N}}
\nc{\frakO}{{\mathfrak O}}
\nc{\frakP}{{\mathfrak P}}
\nc{\frakQ}{{\mathfrak Q}}
\nc{\frakR}{{\mathfrak R}}
\nc{\frakS}{{\mathfrak S}}
\nc{\frakT}{{\mathfrak T}}
\nc{\frakU}{{\mathfrak U}}
\nc{\frakV}{{\mathfrak V}}
\nc{\frakW}{{\mathfrak W}}
\nc{\frakX}{{\mathfrak X}}
\nc{\frakY}{{\mathfrak Y}}
\nc{\frakZ}{{\mathfrak Z}}
\nc{\bbA}{{\mathbb A}}
\nc{\bbB}{{\mathbb B}}
\nc{\bbC}{{\mathbb C}}
\nc{\bbD}{{\mathbb D}}
\nc{\bbE}{{\mathbb E}}
\nc{\bbF}{{\mathbb F}}
\nc{\bbG}{{\mathbb G}}
\nc{\bbH}{{\mathbb H}}
\nc{\bbI}{{\mathbb I}}
\nc{\bbJ}{{\mathbb J}}
\nc{\bbK}{{\mathbb K}}
\nc{\bbL}{{\mathbb L}}
\nc{\bbM}{{\mathbb M}}
\nc{\bbN}{{\mathbb N}}
\nc{\bbO}{{\mathbb O}}
\nc{\bbP}{{\mathbb P}}
\nc{\bbQ}{{\mathbb Q}}
\nc{\bbR}{{\mathbb R}}
\nc{\bbS}{{\mathbb S}}
\nc{\bbT}{{\mathbb T}}
\nc{\bbU}{{\mathbb U}}
\nc{\bbV}{{\mathbb V}}
\nc{\bbW}{{\mathbb W}}
\nc{\bbX}{{\mathbb X}}
\nc{\bbY}{{\mathbb Y}}
\nc{\bbZ}{{\mathbb Z}}
\nc{\calA}{{\mathcal A}}
\nc{\calB}{{\mathcal B}}
\nc{\calC}{{\mathcal C}}
\nc{\calD}{{\mathcal D}}
\nc{\calE}{{\mathcal E}}
\nc{\calF}{{\mathcal F}}
\nc{\calG}{{\mathcal G}}
\nc{\calH}{{\mathcal H}}
\nc{\calI}{{\mathcal I}}
\nc{\calJ}{{\mathcal J}}
\nc{\calK}{{\mathcal K}}
\nc{\calL}{{\mathcal L}}
\nc{\calM}{{\mathcal M}}
\nc{\calN}{{\mathcal N}}
\nc{\calO}{{\mathcal O}}
\nc{\calP}{{\mathcal P}}
\nc{\calQ}{{\mathcal Q}}
\nc{\calR}{{\mathcal R}}
\nc{\calS}{{\mathcal S}}
\nc{\calT}{{\mathcal T}}
\nc{\calU}{{\mathcal U}}
\nc{\calV}{{\mathcal V}}
\nc{\calW}{{\mathcal W}}
\nc{\calX}{{\mathcal X}}
\nc{\calY}{{\mathcal Y}}
\nc{\calZ}{{\mathcal Z}}

\nc{\olO}{\overline{\calO}}

\nc{\al}{{\alpha}} 
\nc{\be}{{\beta}}
\nc{\ga}{{\gamma}} \nc{\Ga}{{\Gamma}}
\nc{\ve}{{\varepsilon}} 
\nc{\la}{{\lambda}} \nc{\La}{{\Lambda}}
\nc{\om}{\omega} \nc{\Om}{\Omega} 
\nc{\sig}{{\sigma}} \nc{\Sig}{{\Sigma}}

\nc{\tnb}{\psi_{\rm tame}}

\nc{\op}{{\on{op}}}
\nc{\ad}{{\on{ad}}}
\nc{\alg}{{\on{alg}}}
\nc{\Ad}{{\on{Ad}}}
\nc{\Adm}{{\on{Adm}}} \nc{\aff}{{\on{aff}}}
\nc{\Aff}{{\mathbf{Aff}}}
\nc{\Aut}{{\on{Aut}}}
\nc{\Bun}{{\on{Bun}}}
\nc{\cha}{{\on{char}}}
\nc{\der}{{\on{der}}}
\nc{\Der}{{\on{Der}}}
\nc{\diag}{{\on{diag}}}
\nc{\End}{{\on{End}}}
\nc{\Fl}{{\calF\!\ell}}
\nc{\Gal}{{\on{Gal}}}
\nc{\Gr}{{\on{Gr}}}
\nc{\rH}{{\on{H}}}
\nc{\Hom}{{\on{Hom}}}
\nc{\IC}{{\on{IC}}}
\nc{\id}{{\on{id}}}
\nc{\Id}{{\on{Id}}}
\nc{\ind}{{\on{ind}}}
\nc{\Ind}{{\on{Ind}}}
\nc{\Lie}{{\on{Lie}}}
\nc{\Pic}{{\on{Pic}}}
\nc{\pr}{{\on{pr}}}
\nc{\Res}{{\on{Res}}}
\nc{\res}{{\on{res}}} \nc{\Sat}{{\on{Sat}}}
\nc{\s}{{\on{sc}}}
\nc{\drv}{{\on{der}}}
\nc{\sgn}{{\on{sgn}}}
\nc{\Spec}{{\on{Spec}}}
\nc{\Sph}{\on{Sph}}
\nc{\St}{{\on{St}}}
\nc{\tr}{{\on{tr}}}
\nc{\Tr}{{\on{Tr}}}
\nc{\Mod}{{\mathrm{-Mod}}}
\nc{\Hilb}{{\on{Hilb}}} 
\nc{\Ext}{{\on{Ext}}} 
\nc{\vs}{{\on{Vec}}}

\nc{\GL}{{\on{GL}}}
\nc{\GSp}{{\on{GSp}}}
\nc{\gl}{{\frakg\frakl}}
\nc{\SL}{{\on{SL}}} 
\nc{\SU}{{\on{SU}}} 
\nc{\SO}{{\on{SO}}}

\nc{\Conv}{{\on{Conv}}}
\nc{\Rep}{{\on{Rep}}}
\nc{\Dom}{{\on{Dom}}}
\nc{\red}{{\on{red}}}
\nc{\act}{\on{act}}

\nc{\str}{{\on{-}}} 
\nc{\os}{\overline{s}}
\nc{\oeta}{\overline{eta}}

\nc{\hookto}{\hookrightarrow}
\nc{\longto}{\longrightarrow}
\nc{\leftto}{\leftarrow}

\nc{\bFl}{{\overline{\Fl}}} 
\nc{\bU}{{\overline{U}}}
\nc{\wGr}{{\widetilde{\Gr}}}
\nc{\cGr}{\calG\! r}
\nc{\ohtimes}{\stackrel{!}{\otimes}}

\nc{\bslash}{\backslash}
\nc{\algQl}{{\bar{\bbQ}_\ell}}
\nc{\sF}{{\bar{F}}}
\nc{\sk}{{\bar{k}}}
\nc{\cont}{\on{c}}
\nc{\boxtilde}{\widetilde{\boxtimes}}
\nc{\vstar}{{\varhexstar}}
\nc{\supp}{\on{supp}}

\nc{\blt}{\bullet} 
\nc{\Spf}{\on{Spf}} 

\nc{\pot}[1]{ [\hspace{-0,5mm}[ {#1} ]\hspace{-0,5mm}] }
\nc{\rpot}[1]{ (\hspace{-0,7mm}( {#1} )\hspace{-0,7mm}) }

\nc{\defined}{\hspace{0.1cm}\stackrel{\text{\tiny def}}{=}\hspace{0.1cm}}

\begin{document}

\title[Geometric Satake]{A new approach to the geometric Satake equivalence}
\author[T. Richarz]{by Timo Richarz}

\address{Timo Richarz: Mathematisches Institut der Universit\"at Bonn, Endenicher Allee 60, 53115 Bonn, Germany}
\email{richarz@math.uni-bonn.de}

\maketitle

\begin{abstract}
I give another proof of the geometric Satake equivalence from I. Mirkovi\'c and K. Vilonen \cite{MV} over a separably closed field. Over a not necessarily separably closed field, I obtain a canonical construction of the Galois form of the full $L$-group.
\end{abstract}

\tableofcontents
\setcounter{section}{-1}

\section{Introduction}
Connected reductive groups over separably closed fields are classified by their root data. These come in pairs: to every root datum, there is associated its dual root datum and vice versa. Hence, to every connected reductive group $G$, there is associated its dual group $\hat{G}$. Following Drinfeld's geometric interpretation of Langlands' philosophy, Mirkovi\'c and Vilonen \cite{MV} show that the representation theory of $\hat{G}$ is encoded in the geometry of an ind-scheme canonically associated to $G$ as follows.  

Let $G$ be a connected reductive group over a separably closed field $F$. The \emph{loop group} $LG$ is the group functor on the category of $F$-algebras
\[LG:R \longmapsto G(R\rpot{t}).\]
The \emph{positive loop group} $L^+G$ is the group functor
\[L^+G:R \longmapsto G(R\pot{t}).\]
Then $L^+G\subset LG$ is a subgroup functor, and the fpqc-quotient $\Gr_G=LG/L^+G$ is called the \emph{affine Grassmannian}. It is representable by an ind-projective ind-scheme (= inductive limit of projective schemes). Now fix a prime $\ell\neq\cha(F)$, and consider the category $P_{L^+G}(\Gr_G)$ of $L^+G$-equivariant $\ell$-adic perverse sheaves on $\Gr_G$. This is a $\algQl$-linear abelian category with simple objects as follows. Fix $T\subset B\subset G$ a maximal torus contained in a Borel. For every cocharacter $\mu$, denote by 
\[\olO_\mu\defined \overline{L^+G\cdot t^\mu}\]
the reduced $L^+G$-orbit closure of $t^\mu\in T(F\rpot{t})$ inside $\Gr_G$. Then $\olO_\mu$ is a projective variety over $F$. Let $\IC_\mu$ be the intersection complex of $\olO_\mu$. The simple objects of $P_{L^+G}(\Gr_G)$ are the $\IC_\mu$'s where $\mu$ ranges over the set of dominant cocharacters $X^\vee_+$. Furthermore, the category $P_{L^+G}(\Gr_G)$ is equipped with an inner product: to every $\calA_1,\calA_2\in P_{L^+G}(\Gr_G)$, there is associated a perverse sheaf $\calA_1\star\calA_2\in P_{L^+G}(\Gr_G)$ called the \emph{convolution product} of $\calA_1$ and $\calA_2$ (cf. \S \ref{convprodsec} below). Denote by 
\[\om(\str)\defined\bigoplus_{i\in\bbZ}R^i\Ga(\Gr_G,\str):\,P_{L^+G}(\Gr_G)\longto\vs_{\algQl}\] 
the global cohomology functor with values in the category of finite dimensional $\algQl$-vector spaces. Fix a pinning of $G$, and let $\hat{G}$ be the Langlands dual group over $\algQl$, i.e. the reductive group over $\algQl$ whose root datum is dual to the root datum of $G$. 

\begin{thm}\label{MVthm}
\emph{(i)} The pair $(P_{L^+G}(\Gr_G),\star)$ admits a unique symmetric monoidal structure such that the functor $\om$ is symmetric monoidal. \smallskip \\
\emph{(ii)} The functor $\omega$ is a faithful exact tensor functor, and induces via the Tannakian formalism an equivalence of tensor
categories
\begin{align*}
(P_{L^+G}(\Gr_G),\star)&\;\overset{\simeq}{\longto}\; (\Rep_{\algQl}(\hat{G}),\otimes)\\
\calA &\;\longmapsto\;\om(\calA),
\end{align*}
which is uniquely determined up to inner automorphisms of $\hat{G}$ by the property that $\omega(\text{IC}_\mu)$ is the irreducible representation of highest weight $\mu$ (for the dual torus $\hat{T}$).
\end{thm}

In the case $F=\mathbb{C}$, this reduces to the theorem of Mirkovi\'c and Vilonen \cite{MV} for coefficient fields of characteristic $0$. The drawback of our method is the restriction to $\algQl$-coefficients. Mirkovic and Vilonen are able to establish a geometric Satake equivalence with coefficients in any Noetherian ring of finite global dimension (in the analytic topology). I give a proof of the theorem over any separably closed field $F$ using $\ell$-adic perverse sheaves. My proof is different from the one of Mirkovi\'c and Vilonen. It proceeds in two main steps as follows. 

In the first step I show that the pair $(P_{L^+G}(\text{Gr}_G),\star)$ is a symmetric monoidal category. This relies on the \emph{Beilinson-Drinfeld Grassmannians} and the comparison of the convolution product with the \emph{fusion product} via Beilinson's construction of the nearby cycles functor. The method is related to ideas of Gaitsgory \cite{Ga} which were extended by Reich in \cite{RR}. Here the fact that the convolution of two perverse sheaves is perverse is deduced from the fact that nearby cycles preserve perversity.

The second step is the identification of the group of tensor automorphisms $\underline{\text{Aut}}^\star(\omega)$ with the reductive group $\hat{G}$. I use a theorem of Kazhdan, Larsen and Varshavsky \cite{KLV} which states that the root datum of a split reductive group can be reconstructed from the Grothendieck semiring of its algebraic representations. The reconstruction of the root datum relies on the PRV-conjecture proven by Kumar \cite{Kumar}. I prove the following geometric analogue of the PRV-conjecture. 

\begin{thm}[Geometric analogue of the PRV-Conjecture]
Denote by $W=W(G,T)$ the Weyl group. Let $\mu_1,\ldots,\mu_n\in X_+^\vee$ be dominant coweights. Then, for every $\la\in X_+^\vee$ of the form $\la=\nu_1+\ldots+\nu_k$ with $\nu_i\in W\mu_i$ for $i=1,\ldots,k$, the perverse sheaf $\IC_{\la}$ appears as a direct summand in the convolution product $\text{IC}_{\mu_1}\star\ldots\star\text{IC}_{\mu_n}$.
\end{thm}

Using this theorem and the method in \cite{KLV}, I show that the Grothendieck semirings of $P_{L^+G}(\text{Gr}_G)$ and $\Rep_{\algQl}(\hat{G})$ are isomorphic. Hence, the root data of $\underline{\text{Aut}}^\star(\omega)$ and $\hat{G}$ are the same. This shows that $\underline{\text{Aut}}^\star(\omega)\simeq\hat{G}$ uniquely up to inner automorphism of $\hat{G}$.

If $F$ is not neccessarily separably closed, we are able to apply Galois descent to reconstruct the full $L$-group. Fix a separable closure $\sF$ of $F$, and denote by $\Ga=\Gal(\sF/F)$ the absolute Galois group. Let $^LG=\hat{G}(\algQl)\rtimes\Ga$ be the Galois form of the full $L$-group with respect to some pinning.

\begin{thm}\label{fullLthm}
The functor $\calA\mapsto\om(\calA_{\sF})$ induces an equivalence of abelian tensor categories
\[(P_{L^+G}(\Gr_G),\star)\;\simeq\;(\Rep_{\algQl}^c(^LG),\otimes),\]
where $\Rep_{\algQl}^c(^LG)$ is the full subcategory of the category of finite dimensional continuous $\ell$-adic representations of $^LG$ such that the restriction to $\hat{G}(\algQl)$ is algebraic.
\end{thm}

We outline the structure of the paper. In \S \ref{satcatsec} we introduce the Satake category $P_{L^+G}(\Gr_G)$. Appendix \ref{pervapp} supplements the definition of $P_{L^+G}(\Gr_G)$ and explains some basic facts on perverse sheaves on ind-schemes as used in the paper. In \S \ref{convprodsec}-\S \ref{tannakastr} we clarify the tensor structure of the tuple $(P_{L^+G}(\Gr_G),\star)$, and show that it is neutralized Tannakian with fiber functor $\om$. Section \ref{tannakaeq} is devoted to the identification of the dual group. This section is supplemented by Appendix \ref{reconapp} on the reconstruction of root data from the Grothendieck semiring of algebraic representations. The reader who is just interested in the case of an algebraically closed ground field may assume $F$ to be algebraically closed throughout \S\ref{satcatsec}-\S\ref{tannakaeq}. The last section \S \ref{galoisdescent} is concerned with Galois descent and the reconstruction of the full $L$-group.

\begin{ack}
First of all I thank my advisor M. Rapoport for his steady encouragement and advice during the process of writing. I am grateful to the stimulating working atmosphere in Bonn and for the funding by the Max-Planck society.
\end{ack}

\section{The Satake Category}\label{satcatsec} 
Let $G$ a connected reductive group over any field $F$. The \emph{loop group} $LG$ is the group functor on the category of $F$-algebras
\[LG:R \longmapsto G(R\rpot{t}).\]
The \emph{positive loop group} $L^+G$ is the group functor
\[L^+G:R \longmapsto G(R\pot{t}).\]
Then $L^+G\subset LG$ is a subgroup functor, and the fpqc-quotient $\Gr_G=LG/L^+G$ is called the \emph{affine Grassmannian} (associated to $G$ over $F$). 

\begin{lem} \label{basicgrass}
The affine Grassmannian $\Gr_G$ is representable by an ind-projective strict ind-scheme over $F$. It represents the functor which assigns to every $F$-algebra $R$ the set of isomorphism classes of pairs $(\calF,\beta)$, where $\calF$ is a $G$-torsor over $\Spec(R\pot{t})$ and $\beta$ a trivialization of $\calF[\frac{1}{t}]$ over $\Spec(R\rpot{t})$.
\end{lem}

We postpone the proof of Lemma \ref{basicgrass} to Section \ref{bdgrass} below. For every $i\geq0$, let $G_i$ denote $i$-th jet group, given for any $F$-algebra $R$ by $G_i:R\mapsto G(R[t]/t^{i+1})$. Then $G_i$ is representable by a smooth connected affine group scheme over $F$ and, as fpqc-sheaves,
\[L^+G \;\simeq\; \varprojlim_{i}G_i.\]
In particular, if $G$ is non trivial, then $L^+G$ is not of finite type over $F$. The positive loop group $L^+G$ operates on $\Gr_G$ and, for every orbit $\calO$, the $L^+G$-action factors through $G_i$ for some $i$. Let $\olO$ denote the reduced closure of $\calO$ in $\Gr_G$, a projective $L^+G$-stable subvariety. This presents the reduced locus as the direct limit of $L^+G$-stable subvarieties
\[(\Gr_G)_{\red} \;=\; \varinjlim_{\calO} \olO,\] 
where the transition maps are closed immersions. 

Fix a prime $\ell\not=\cha(F)$, and denote by $\bbQ_\ell$ the field of $\ell$-adic numbers with algebraic closure $\algQl$. For any separated scheme $T$ of finite type over $F$, we consider the bounded derived category $D_c^b(T,\algQl)$ of constructible $\ell$-adic complexes on $T$, and its abelian full subcategory $P(T)$ of $\ell$-adic perverse sheaves. If $H$ is a connected smooth affine group scheme acting on $T$, then let $P_H(T)$ be the abelian subcategory of $P(T)$ of $H$-equivariant objects with $H$-equivariant morphisms. We refer to Appendix \ref{pervapp} for an explanation of these concepts. 

The category of $\ell$-adic perverse sheaves $P(\Gr_G)$ on the affine Grassmannian is the direct limit
\[P(\Gr_G) \defined \varinjlim_{\calO}P(\olO),\]
 which is well-defined, since all transition maps are closed immersions, cf. Appendix \ref{pervapp}. 
  
\begin{dfn} 
 The \emph{Satake category} is the category of $L^+G$-equivariant $\ell$-adic perverse sheaves on the affine Grassmannian $\Gr_G$   
\[P_{L^+G}(\Gr_G)\defined\varinjlim_{\calO}P_{L^+G}(\olO),\]
where $\calO$ ranges over the $L^+G$-orbits.
\end{dfn}

The Satake category $P_{L^+G}(\Gr_G)$ is an abelian $\algQl$-linear category, cf. Appendix \ref{pervapp}.

\section{The Convolution Product}\label{convprodsec}
 We are going to equip the category $P_{L^+G}(\Gr_G)$ with a tensor structure. Let 
\[\str\star\str:P(\Gr_G)\times P_{L^+G}(\Gr_G)\longto D_c^b(\Gr_G,\algQl)\] 
be the convolution product with values in the derived category. We recall its definition \cite[\S 2]{NP}. Consider the following diagram of ind-schemes
\begin{equation}\label{convdiag}
\Gr_G\times\Gr_G\stackrel{p}{\longleftarrow} LG\times\Gr_G\stackrel{q}{\longto} LG\times^{L^+G}\Gr_G\stackrel{m}{\longto}\Gr_G.
\end{equation}
Here $p$ (resp. $q$) is a right $L^+G$-torsor with respect to the $L^+G$-action on the left factor (resp. the diagonal action).The $LG$-action on $\Gr_G$ factors through $q$, giving rise to the morphism $m$.

For perverse sheaves $\calA_1,\calA_2$ on $\Gr_G$, their box product $\calA_1\boxtimes\calA_2$ is a perverse sheaf on $\Gr_G\times\Gr_G$. If $\calA_2$ is $L^+G$-equivariant, then there is a unique perverse sheaf $\calA_1\boxtilde\calA_2$ on $LG\times^{L^+G}\Gr_G$ such that there is an isomorphism equivariant for the diagonal $L^+G$-action\footnote{Though $LG$ is not of ind-finite type, we use Lemma \ref{twgloboxlem} below to define $\calA_1\boxtilde\calA_2$.}
\[p^*(\calA_1\boxtimes\calA_2)\simeq q^*(\calA_1\boxtilde\calA_2).\]
Then the convolution is defined as $\calA_1\star\calA_2\defined m_*(\calA_1\boxtilde\calA_2)$. 

\begin{thm} \label{monoidalthm}
\emph{(i)} For perverse sheaves $\calA_1,\calA_2$ on $\Gr_G$ with $\calA_2$ being $L^+G$-equivariant, their convolution $\calA_1\star\calA_2$ is a perverse sheaf. If $\calA_1$ is also $L^+G$-equivariant, then $\calA_1\star\calA_2$ is $L^+G$-equivariant. \smallskip \\
\emph{(ii)} Let $\sF$ be a separable closure of $F$. The convolution product is a bifunctor
\[\str\star\str:P_{L^+G}(\Gr_G)\times P_{L^+G}(\Gr_G) \longto P_{L^+G}(\Gr_G),\]
and $(P_{L^+G}(\Gr_G),\star)$ has a unique structure of a symmetric monoidal category such that the cohomology functor with values in finite dimensional $\algQl$-vector spaces 
\[\bigoplus_{i\in\bbZ}R^i\Ga(\Gr_{G,\sF},(\str)_{\sF})\!: P_{L^+G}(\Gr_G) \longto \vs_{\algQl}\]
is symmetric monoidal.
\end{thm}

Part (i) and (ii) of Theorem \ref{monoidalthm} are proved simultaneously in Subsection \ref{smstr} below using universally locally acyclic perverse sheaves (cf. Subsection \ref{ULA} below) and a global version of diagram \eqref{convdiag} which we introduce in the next subsection.  

\subsection{Beilinson-Drinfeld Grassmannians}\label{bdgrass} Let $X$ a smooth geometrically connected curve over $F$. For any $F$-algebra $R$, let $X_R=X\times\Spec(R)$. Denote by $\Sig$ the moduli space of relative effective Cartier divisors on $X$, i.e. the fppf-sheaf associated with the functor on the category of $F$-algebras
\[R \;\longmapsto\; \{D\subset X_R \;\;\text{relative effective Cartier divisor}\}.\]

\begin{lem}\label{cartmodlem}
The $fppf$-sheaf $\Sig$ is represented by the disjoint union of fppf-quotients $\coprod_{n\geq1}X^n/S_n$,
where the symmetric group $S_n$ acts on $X^n$ by permuting its coordinates.
\end{lem}
\hfill\ensuremath{\Box}

\begin{dfn}
The \emph{Beilinson-Drinfeld Grassmannian (associated to $G$ and $X$)} is the functor $\cGr=\cGr_{G,X}$ on the category of $F$-algebras which assings to every $R$ the set of isomorphism classes of triples $(D,\calF,\beta)$ with 
\begin{align*}
\begin{cases}
& D\in \Sig(R) \;\text{a relative effective Cartier divisor}; \\
& \calF \; \text{a} \; G\text{-torsor on} \; X_R; \\
& \be: \calF|_{X_R\bslash D}\stackrel{\simeq}{\to} \calF_0|_{X_R\bslash D} \;\text{a trivialisation}, 
\end{cases}
\end{align*}
where $\calF_0$ denotes the trivial $G$-torsor. The projection $\cGr\to\Sig$, $(D,\calF,\be)\mapsto D$ is a morphism of functors.
\end{dfn}

\begin{lem}\label{globalgrasslem}
The Beilinson-Drinfeld Grassmannian $\cGr=\cGr_{G,X}$ associated to a reductive group $G$ and a smooth curve $X$ is representable by an ind-proper strict ind-scheme over $\Sig$. 
\end{lem}
\begin{proof}
This is proven in \cite[Appendix A.5.]{Ga}. We sketch the argument. If $G=\GL_n$, consider the functor $\cGr_{(m)}$ parametrizing  
\[J\subset \calO^n_{X_R}(-m\cdot D)/\calO^n_{X_R}(m\cdot D),\]
where $J$ is a coherent $\calO_{X_R}$-submodule such that $\calO_{X_R}(-m\cdot D)/J$ is flat over $R$. By the theory of Hilbert schemes, the functor $\cGr_{(m)}$ is representable by a proper scheme over $\Sig$. For $m_1<m_2$, there are closed immersions $\cGr_{(m_1)}\hookto \cGr_{(m_2)}$. Then as fpqc-sheaves
\[\varinjlim_m\,\cGr_{(m)}\;\stackrel{\simeq}{\longto}\; \cGr.\]
For general reductive $G$, choose an embedding $G\hookto\GL_n$. Then the fppf-quotient $\GL_n/G$ is affine, and the natural morphism $\cGr_{G}\to\cGr_{\GL_n}$ is a closed immersion. The ind-scheme structure of $\cGr_G$ does not depend on the choosen embedding $G\hookto\GL_n$. This proves the lemma. 
\end{proof} 

Now we define a global version of the loop group. For every $D\in\Sig(R)$, the formal completion of $X_R$ along $D$ is a formal affine scheme. We denote by $\hat{\calO}_{X,D}$ its underlying $R$-algebra. Let $\hat{D}=\Spec(\hat{\calO}_{X,D})$ be the associated affine scheme over $R$. Then $D$ is a closed subscheme of $\hat{D}$, and we set $\hat{D}^o=\hat{D}\backslash D$. The \emph{global loop group} is the group functor on the category of $F$-algebras
\[\calL G:R\mapsto \{(s,D)\;|\;D\in\Sig(R),\; s\in G(\hat{D}^o)\}.\]
The \emph{global positive loop group} is the group functor
\[\calL^+ G:R\mapsto \{(s,D)\;|\; D\in\Sig(R),\; s\in G(\hat{D})\}.\]
Then $\calL^+ G\subset\calL G$ is a subgroup functor over $\Sig$.

\begin{lem}\label{connectlem}
\emph{(i)} The global loop group $\calL G$ is representable by an ind-group scheme over $\Sig$. It represents the functor on the category of $F$-algebras which assigns to every $R$ the set of isomorphism classes of quadruples $(D,\calF,\be,\sig)$, where $D\in\Sig(R)$, $\calF$ is a $G$-torsor on $X_R$, $\be:\calF\stackrel{\simeq}{\to}\calF_0$ is a trivialisation over $X_R\bslash D$ and $\sig:\calF_0\stackrel{\simeq}{\to}\calF|_{\hat{D}}$ is a trivialisation over $\hat{D}$. \smallskip \\
\emph{(ii)} The global positive loop group $\calL^+G$ is representable by an affine group scheme over $\Sig$ with geometrically connected fibers. \smallskip \\
\emph{(iii)} The projection $\calL G\to\cGr_G$, $(D,\calF,\be,\sig)\to(D,\calF,\be)$ is a right $\calL^+G$-torsor, and induces an isomorphism of $fpqc$-sheaves over $\Sig$
\[\;\calL G/\calL^+G\stackrel{\simeq}{\longto}\; \cGr_G.\]
\end{lem}
\begin{proof} 
Note that fppf-locally on $R$ every $D\in\Sig(R)$ is of the form $V(f)$. Then the moduli description in (i) follows from the descent lemma of Beauville-Laszlo \cite{BL} (cf. \cite[Proposition 3.8]{LS}). The ind-representability follows from part (ii) and (iii). This proves (i).\\
For any $D\in\Sig(R)$ denote by $D^{(i)}$ its $i$-th infinitesimal neighbourhood in $X_R$. Then $D^{(i)}$ is finite over $R$, and the Weil restriction $\Res_{D^{(i)}/R}(G)$ is representable by a smooth affine group scheme with geometrically connected fibers. For $i\leq j$, there are affine transition maps $\Res_{D^{(j)}/R}(G)\to\Res_{D^{(i)}/R}(G)$ with geometrically connected fibers. Hence, $\varprojlim_i\Res_{D^{(i)}/R}(G)$ is an affine scheme, and the canonical map 
\[\calL^+G\times_{\Sig,D}\Spec(R)\longto \varprojlim_i\Res_{D^{(i)}/R}(G)\] 
is an isomorphism of fpqc-sheaves. This proves (ii).\\
To prove (iii), the crucial point is that after a faithfully flat extension $R\to R'$ a $G$-torsor $\calF$ on $\hat{D}$ admits a global section. Indeed, $\calF$ admits a $R'$-section which extends to $\hat{D}_{R'}$ by smoothness and Grothendieck's algebraization theorem. This finishes (iii).
\end{proof}

\begin{rmk}
The connection with the affine Grassmannian $\Gr_G$ is as follows. Lemma \ref{cartmodlem} identifies $X$ with a connected component of $\Sig$. Choose a point $x\in X(F)$ considered as an element $D_x\in\Sig(F)$. Then $\hat{D}_x\simeq\Spec(F\pot{t})$, where $t$ is a local parameter of $X$ in $x$. Under this identification, there are isomorphisms of fpqc-sheaves 
\begin{align*}
\calL G_x&\simeq LG \\
\calL^+ G_x&\simeq L^+G \\
\cGr_{G,x}&\simeq\Gr_G.
\end{align*}
Using the theory of Hilbert schemes, the proof of Lemma \ref{globalgrasslem} also implies that $\Gr_{\GL_n}$, and hence $\Gr_G$ is ind-projective. This proves Lemma \ref{basicgrass} above.
\end{rmk}

By Lemma \ref{connectlem} (iii), the global positive loop groop $\calL^+G$ acts on $\cGr$ from the left. For $D\in\Sig(R)$ and $(D,\calF,\be)\in\cGr_G(R)$, denote the action by
\[((g,D),(\calF,\be,D))\longmapsto(g\calF,g\be,D).\]

\begin{cor}
The $\calL^+G$-orbits on $\cGr$ are of finite type and smooth over $\Sigma$. 
\end{cor}
\begin{proof}
Let $D\in\Sig(R)$. It is enough to prove that the action of 
\[\calL^+G\times_{\Sig,D}\Spec(R)\;\simeq\; \varprojlim_i\Res_{D^{(i)}/R}(G)\]
on $\cGr\times_{\Sig,D}\Spec(R)$ factors over $\Res_{D^{(i)}/R}(G)$ for some $i>\!>0$. Choose a faithful representation $\rho:G\to \GL_n$, and consider the corresponding closed immersion $\cGr_G\to \cGr_{\GL_n}$. This reduces us to the case $G=\GL_n$. In this case, the $\cGr_{(m)}$'s (cf. proof of Lemma \ref{globalgrasslem}) are $\calL^+\GL_n$ stable, and it is easy to see that the action on $\cGr_{(m)}$ factors through $\Res_{D^{(2m)}/R}(\GL_n)$. This proves the corollary. 
\end{proof}

Now we globalize the convolution morphism $m$ from diagram \eqref{convdiag} above. The moduli space $\Sig$ of relative effective Cartier divisors has a natural monoid structure 
\begin{align*}
\str\cup\str:\Sig\times\Sig&\longto\Sig\\
(D_1,D_2)&\longmapsto D_1\cup D_2.
\end{align*}
The key definition is the following.

\begin{dfn}
For $k\geq 1$, the \emph{$k$-fold convolution Grassmannian} $\tilde{\cGr}_k$ is the functor on the category of $F$-algebras which associates to every $R$ the set of isomorphism classes of tuples $((D_i,\calF_i,\be_i)_{i=1,\ldots,k})$ with
\begin{align*}
\begin{cases}
& D_i\in \Sig(R) \;\text{relative effective Cartier divisors,}\;i=1,\ldots,k; \\
& \calF_i \;\text{are}\; G\text{-torsors on} \; X_R; \\
& \be_i: \calF_i|_{X_R\bslash D_i}\stackrel{\simeq}{\to} \calF_{i-1}|_{X_R\bslash D_i} \;\text{isomorphisms,}\;i=1,\ldots,k, 
\end{cases}
\end{align*}
where $\calF_0$ is the trivial $G$-torsor. The projection $\tilde{\cGr}_k\to\Sig^k$, $((D_i,\calF_i,\be_i)_{i=1,\ldots,k})\mapsto ((D_i)_{i=1,\ldots,k})$ is a morphism of functors.
\end{dfn} 

\begin{lem}
For $k\geq 1$, the $k$-fold convolution Grassmannian $\tilde{\cGr}_k$ is representable by a strict ind-scheme which is ind-proper over $\Sig^k$.
\end{lem}
\begin{proof}
The lemma follows by induction on $k$. If $k=1$, then $\tilde{\cGr}_k=\cGr$. For $k>1$, consider the projection
\begin{align*}
p :\;\tilde{\cGr}_k&\longto \tilde{\cGr}_{k-1}\times\Sig \\
((D_i,\calF_i,\be_i)_{i=1,\ldots,k})&\longmapsto ((D_i,\calF_i,\be_i)_{i=1,\ldots,k-1},D_k). 
\end{align*}
Then the fiber over a $R$-point $((D_i,\calF_i,\be_i)_{i=1,\ldots,k-1},D_k)$ is
\[p^{-1}(((D_i,\calF_i,\be_i)_{i=1,\ldots,k-1},D_k))\;\simeq\;\calF_{k-1}\times^G\cGr,\]
which is ind-proper. This proves the lemma.
\end{proof}

For $k\geq1$, there is the \emph{$k$-fold global convolution morphism} 
\begin{align*}
m_k: \tilde{\cGr}_k&\longto\cGr \\
((D_i,\calF_i,\be_i)_{i=1,\ldots,k})&\longmapsto (D,\calF_k,\be_1|_{X_R\bslash D}\circ\ldots\circ\be_k|_{X_R\bslash D}),
\end{align*}
where $D=D_1\cup\ldots\cup D_k$. This yields a commutative diagram of ind-schemes
\[\begin{tikzpicture} 
\matrix(a)[matrix of math nodes, 
row sep=1.5em, column sep=2em, 
text height=1.5ex, text depth=0.45ex] 
{\tilde{\cGr}_k & \cGr \\ 
\Sig^k & \Sig,\\}; 
\path[->](a-1-1) edge node[above] {$m_k$} (a-1-2); 
\path[->](a-2-1) edge node[above] {$\cup$} (a-2-2); 
\path[->](a-1-1) edge node[right] {} (a-2-1);
\path[->](a-1-2) edge node[right] {} (a-2-2); 
\end{tikzpicture}\]
i.e., regarding $\tilde{\cGr}_k$ as a $\Sig$-scheme via $\Sig^k\to\Sig$, $(D_i)_i\mapsto \cup_iD_i$, the morphism $m_k$ is a morphism of $\Sig$-ind-schemes. The global positive loop group $\calL^+G$ acts on $\tilde{\cGr}_k$ over $\Sigma$ as follows: let $(D_i,\calF_i,\be_i)_i\in\tilde{\cGr}_k(R)$ and $g\in G(\hat{D})$ with $D=\cup_iD_i$. Then the action is defined as
\[((g,D),(D_i,\calF_i,\be_i)_i)\longmapsto(D_i, g\calF_i, g\be_ig^{-1})_i.\]
 
\begin{cor}
The morphism $m_k:\tilde{\cGr}_k\to\cGr$ is a $\calL^+G$-equivariant morphism of ind-proper strict ind-schemes over $\Sig$.
\end{cor}
\begin{proof}
The $\calL^+G$-equivariance is immediate from the definition of the action. Note that $\Sig^k\stackrel{\cup}{\to}\Sig$ is finite, and hence $\tilde{\cGr}_k$ is an ind-proper strict ind-scheme over $\Sig$. This proves the corollary.
\end{proof}

Now we explain the global analogue of the $L^+G$-torsors $p$ and $q$ from \eqref{convdiag}. For $k\geq 1$, let  $\tilde{\calL} G_k$ be the functor on the category of $F$-algebras which associates to every $R$ the set of isomorphism classes of tuples $((D_i,\calF_i,\be_i)_{i=1,\ldots,k},(\sig_i)_{i=2,\ldots,k})$ with
\begin{align*}
\begin{cases}
& D_i\in \Sig(R), \;i=1,\ldots,k; \\
& \calF_i \;\text{are}\; G\text{-torsors on} \; X_R; \\
& \be_i: \calF_i|_{X_R\bslash D_i}\stackrel{\simeq}{\to} \calF_0|_{X_R\bslash D_i} \;\text{trivialisations,}\;i=1,\ldots,k; \\
& \sig_i: \calF_0|_{\hat{D}_i}\stackrel{\simeq}{\to}\calF_{i-1}|_{\hat{D}_i},\;i=2,\ldots,k,
\end{cases}
\end{align*}
where $\calF_0$ is the trivial $G$-torsor. There are two natural projections over $\Sig^k$. Let 
\[\calL^+G^{k-1}_\Sig=\Sig^k\times_{\Sig^{k-1}}\calL^+G^{k-1}.\] 
The first projection is given by 
\begin{align*}
p_k:\tilde{\calL} G_k&\longto\cGr^k \\
((D_i,\calF_i,\be_i)_{i=1,\ldots,k},(\sig_i)_{i=2,\ldots,k})&\longmapsto((D_i,\calF_i|_{\hat{D}_i},\be_i)_{i=1,\ldots,k}).
\end{align*}
Then $p_k$ is a right $\calL^+G^{k-1}_\Sig$-torsor for the action on the $\sig_i\text{'s}$. The second projection is given by
\begin{align*}
q_k:\tilde{\calL} G_k&\longto\tilde{\cGr}_k \\
((D_i,\calF_i,\be_i)_{i=1,\ldots,k},(\sig_i)_{i=2,\ldots,k})&\longmapsto((D_i,\calF'_i,\be'_i)_{i=1,\ldots,k}),
\end{align*}
where $\calF_1'=\calF_1$ and for $i\geq2$, the $G$-torsor $\calF'_i$ is defined successively by gluing $\calF_i|_{X_R\bslash D_i}$ to $\calF_{i-1}'|_{\hat{D}_i}$ along $\sig_{i}|_{\hat{D}_i^o}\circ\be_{i}|_{\hat{D}_i^o}$. Then $q_k$ is a right $\calL^+G^{k-1}_\Sig$-torsor for the action given by
\[(((D_i,\calF_i,\be_i)_{i\geq1},(\sig_i)_{i\geq2}),(D_1,(D_i,g_i)_{i\geq2}))\longmapsto((D_1,\calF_1,\be_1),(D_i,g_i^{-1}\calF_i,g_i^{-1}\be_i)_{i\geq2},(\sig_ig_i)_{i\geq2}).\] 
In the following, we consider ind-schemes over $\Sig^k$ as ind-schemes over $\Sig$ via $\Sig^k\to\Sig$. 

\begin{dfn}\label{globalconvdiagdfn}
For every $k\geq1$, the \emph{$k$-fold global convolution diagram} is the diagram of ind-schemes over $\Sig$
\[\cGr^k\stackrel{p_k}{\longleftarrow} \tilde{\calL}G_k\stackrel{q_k}{\longto} \tilde{\cGr}_k\stackrel{m_k}{\longto}\cGr.\]
\end{dfn}

\begin{rmk}\label{usualconvdiag}
Fix $x\in X(F)$, and choose a local coordinate $t$ at $x$. Taking the fiber over $\diag(\{x\})\in X^k(F)$ in the $k$-fold global convolution diagram, then
\[\Gr_G^k\longleftarrow LG^{k-1}\times\Gr_G\longto \underbrace{LG\times^{L^+G}\ldots\times^{L^+G}\Gr_G}_{k\text{-times}}\longto\Gr_G.\]
For $k=2$, we recover diagram \eqref{convdiag}.
\end{rmk}
 
\subsection{Universal Local Acyclicity}\label{ULA} 
The notion of universal local acyclicity (ULA) is used in Reich's thesis \cite{RR}. We recall the definition. Let $S$ be a smooth geometrically connected scheme over $F$, and $f:T\to S$ a separated morphism of finite type. For complexes $\calA_T\in D_c^b(T,\algQl)$, $\calA_S\in D_c^b(S,\algQl)$, there is a natural morphism
\begin{equation} \label{ULAmap}
\calA_T\otimes f^*\calA_S\;\longto\;(\calA_T\ohtimes f^!\calA_S)[2\dim(S)],
\end{equation}  
where $\calA\ohtimes \calB\defined \bbD(\bbD\calA\otimes \bbD\calB)$ for $\calA,\calB\in D_c^b(T,\algQl)$. The morphism \eqref{ULAmap} is constructed as follows. Let $\Ga_f:T\to T\times S$ be the graph of $f$. The projection formula gives a map
\[\Ga_{f,!}(\Ga_f^*(\calA_T\boxtimes\calA_S)\otimes\Ga_f^!\algQl)\;\simeq\;(\calA_T\boxtimes\calA_S)\otimes\Ga_{f,!}\Ga_f^!\algQl\;\longto\;\calA_T\boxtimes\calA_S,\]
and by adjunction a map $\Ga_f^*(\calA_T\boxtimes\calA_S)\otimes\Ga_f^!\algQl\to\Ga_f^!(\calA_T\boxtimes\calA_S)$. Note that 
\[\Ga_f^*(\calA_T\boxtimes\calA_S)\;\simeq\;\calA_T\otimes f^*\calA_S\hspace{1cm}\text{and}\hspace{1cm}\Ga_f^!(\calA_T\boxtimes\calA_S)\;\simeq\;\calA_T\ohtimes f^!\calA_S,\]
using $\bbD(\calA_T\boxtimes\calA_S)\simeq\bbD\calA_T\boxtimes\bbD\calA_S$. Since $S$ is smooth, $\Ga_f$ is a regular embedding, and thus $\Ga_f^!\algQl\simeq\algQl[-2\dim(S)]$. This gives after shifting by $[2\dim(S)]$ the map \eqref{ULAmap}. 

\begin{dfn}
(i) A complex $\calA_T\in D_c^b(T,\algQl)$ is called \emph{locally acyclic with respect to $f$} (f-LA)  if \eqref{ULAmap} is an isomorphism for all $\calA_{S}\in D_c^b(S,\algQl)$. \smallskip \\
(ii) A complex $\calA_T\in D_c^b(T,\algQl)$ is called \emph{universally locally acyclic with respect to $f$} ($f$-ULA) if $f_{S'}^*\calA_T$ is $f_{S'}$-LA for all $f_{S'}=f\times_SS'$ with $S'\to S$ smooth, $S'$ geometrically connected.
\end{dfn}

\begin{rmk} \label{ULArmk}
(i) If $f$ is smooth, then the trivial complex $\calA_T=\algQl$ is $f$-ULA. \\
(ii) If $S=\Spec(F)$ is a point, then every complex $\calA_T\in D_c^b(T,\algQl)$ is $f$-ULA.\\
(iii) The ULA property is local in the smooth topology on $T$.
\end{rmk}

\begin{lem}\label{ULAproperlem}
Let $g:T\to T'$ be a proper morphism of $S$-schemes of finite type. For every ULA complex $\calA_T\in D_c^b(T,\algQl)$, the push forward $g_*\calA_T$ is ULA.
\end{lem}
\begin{proof}
For any morphism of finite type $g:T\to T'$ and any two complexes $\calA_T$, $\calA_{T'}$, we have the projection formulas
\[g_!(\calA_T\otimes g^*\calA_{T'})\;\simeq\; g_!\calA_T\otimes \calA_{T'} \hspace{1cm}\text{and}\hspace{1cm} g_*(\calA_T\ohtimes g^!\calA_{T'})\;\simeq\; g_*\calA_T\ohtimes \calA_{T'}.\]
If $g$ is proper, then $g_*=g_!$, and the lemma follows from an application of the projection formulas and proper base change.
\end{proof}

\begin{thm}[\cite{RR}]\label{mainthm}
Let $D\subset S$ be a smooth Cartier divisor, and consider a cartesian diagram of morphisms of finite type
\[\begin{tikzpicture} 
\matrix(a)[matrix of math nodes, 
row sep=1.5em, column sep=2em, 
text height=1.5ex, text depth=0.45ex] 
{E & T & U \\ 
D & S & S\bslash D.\\}; 
\path[->](a-1-1) edge node[above] {$i$} (a-1-2); 
\path[->](a-2-1) edge  (a-2-2); 
\path[->](a-1-1) edge (a-2-1);
\path[->](a-1-2) edge node[right] {$f$} (a-2-2); 
\path[->](a-1-3) edge node[above] {$j$} (a-1-2); 
\path[->](a-2-3) edge (a-2-2); 
\path[->](a-1-3) edge (a-2-3); 
\end{tikzpicture}\]
Let $\calA$ be a $f$-ULA complex on $T$ such that $\calA|_U$ is perverse. Then: \\
\emph{(i)} There is a functorial isomorphism
\[i^*[-1]\calA \;\simeq\; i^![1]\calA,\]
and both complexes $i^*[-1]\calA$, $i^![1]\calA$ are perverse. Furthermore, the complex $\calA$ is perverse and is the middle perverse extension $\calA\simeq j_{!*}(\calA|_U)$.  \\
\emph{(ii)} The complex $i^*[-1]\calA$ is $f|_E$-ULA.
\end{thm}
\hfill\ensuremath{\Box}

\begin{rmk}
The proof of Theorem \ref{mainthm} uses Beilinson's construction of the unipotent part of the tame nearby cycles as follows. Suppose the Cartier divisor $D$ is principal, this gives a morphism $\varphi:S\to\bbA^1_F$ such that $\varphi^{-1}(\{0\})=S\bslash D$. Let $g=\varphi\circ f$ be the composition. Fix a separable closure $\sF$ of $F$. In SGA VII, Deligne constructs the nearby cycles functor $\psi=\psi_g: P(U)\to P(E_{\sF})$. Let $\tnb$ be the tame nearby cycles, i.e. the invariants under the pro-$p$-part of $\pi_1(\bbG_{m,\sF},1)$. Fix a topological generator $T$ of the maximal prime-$p$-quotient of $\pi_1(\bbG_{m,\sF},1)$. Then $T$ acts on $\tnb$, and there is an exact triangle
\[\tnb\stackrel{T-1}{\longto}\tnb\longto i^*j_*\stackrel{+1}{\longto}\]
Under the action of $T-1$ the nearby cycles decompose as $\tnb\simeq\tnb^{\rm u}\oplus\tnb^{\rm nu}$, where $T-1$ acts nilpotently on $\tnb^{\rm u}$ and invertibly on $\tnb^{\rm un}$. Let $N: \tnb\to\tnb(-1)$ be the logarithm of $T$, i.e. the unique nilpotent operator $N$ such that $T=\exp(\bar{T}N)$ where $\bar{T}$ is the image of $T$ under $\pi_1(\bbG_{m,\sF},1)\twoheadrightarrow\bbZ_{\ell}(1)$. Then for any $a\geq 0$, Beilinson constructs a local system $\calL_a$ on $\bbG_m$ together with a nilpotent operator $N_a$ such that for $\calA_U\in P(U)$ and $a\geq0$ with $N^{a+1}(\tnb^{\rm u}(\calA_U))=0$ there is an isomorphism
\[(\tnb^{\rm u}(\calA_U),N)\;\simeq\;(i^*[-1]j_{!*}(\calA_U\otimes g^*\calL_a)_{\sF},1\otimes N_a).\] 
Set $\Psi_g^{\rm u}(\calA_U)\defined \lim_{a\to\infty}i^*[-1]j_{!*}(\calA_U\otimes g^*\calL_a)$. Then $\Psi_g^{\rm u}:P(U)\to P(E)$ is a functor, and we obtain that $N$ acts trivially on $\tnb^{\rm u}(\calA_U)$ if and only if $\Psi_g^{\rm u}(\calA_U)=i^*[-1]j_{!*}(\calA_U)$. In this case, $\Psi_g^u$ is also defined for non-principal Cartier divisors by the formula $\Psi_g^u= i^*[-1]\circ j_{!*}$.\\
In the situation of Theorem \ref{mainthm} above Reich shows that the unipotent monodromy along $E$ is trivial, and consequently 
\[i^*[-1]\calA\;\simeq\;\Psi_g^u\circ j^*(\calA)\;\simeq\; i^![1]\calA.\].
\end{rmk}

\begin{cor}[\cite{RR}]\label{locsyscor}
Let $\calA$ be a perverse sheaf on $S$ whose support contains an open subset of $S$. Then the following are equivalent:\\
\emph{(i)} The perverse sheaf $\calA$ is ULA with respect to the identity $\text{id}: S\to S$. \\
\emph{(ii)} The complex $\calA[-\dim(S)]$ is a locally constant system, i.e. a lisse sheaf. 
\end{cor}
\hfill\ensuremath{\Box}

We use the universal local acyclicity to show the perversity of certain complexes on the Beilinson-Drinfeld Grassmannian. For every finite index set $I$, there is the quotient map $X^I\to\Sig$ onto a connected component of $\Sig$. Set 
\[\cGr_I\defined\cGr\times_\Sig X^I.\] 
If $I=\{*\}$ has cardinality $1$, we write $\cGr_X=\cGr_I$. 

\begin{rmk}\label{pullbackULAlem}
Let $X=\bbA^1_F$ with global coordinate $t$. Then $\bbG_a$ acts on $X$ via translations. We construct a $\bbG_a$-action on $\cGr$ as follows. For every $x\in\bbG_a(R)$, let $a_x$ be the associated automorphism of $X_R$. If $D\in\Sig(R)$, then we get an isomorphism $a_{-x}:a_xD\to D$. Let $(D,\calF,\be)\in\cGr_G(R)$. Then the $\bbG_a$-action on $\cGr_G\to\Sig$ is given as
\[(D,\calF,\be)\;\longmapsto\;(a^*_{-x}\calF,a^*_{-x}\be,a_xD).\]
Let $\bbG_a$ act diagonally on $X^I$, then the structure morphism $\cGr_I\to X^I$ is $\bbG_a$-equivariant. If $|I|=1$, then by the transitivity of the $\bbG_a$-action on $X$, we get $\cGr_X=\Gr_G\times X$. Let $p:\cGr_X\to\Gr_G$ be the projection. Then for every perverse sheaf $\calA$ on $\Gr_G$, the complex $p^*[1]\calA$ is a ULA perverse sheaf on $\cGr_X$ by Remark \ref{ULArmk} (ii) and the smoothness of $p$. 
\end{rmk}

Now fix a finite index set $I$ of cardinality $k\geq 1$. Consider the base change along $X^I\to \Sig$ of the $k$-fold convolution diagram from Definition \ref{globalconvdiagdfn},
\begin{equation}\label{globalconvdiagcopy}
\prod_{i\in I}\cGr_{X,i}\stackrel{p_I}{\longleftarrow} \tilde{\calL}G_I\stackrel{q_I}{\longto} \tilde{\cGr}_I\stackrel{m_I}{\longto}\cGr_I.
\end{equation}
Now choose a total order $I=\{1,\ldots,k\}$, and set $I^o=I\bslash\{1\}$. Then $p_I$ (resp. $q_I$) is a $\calL^+G_I^o$-torsor, where $\calL^+G_I^o=X^I\times_{X^{I^o}}\calL^+G_{I^o}$. 

Let $\calL^+G_X=\calL^+G\times_{\Sig}X$, and denote by $P_{\calL^+G_X}(\cGr_X)^{\rm ULA}$ the category of $\calL^+G_X$-equivariant ULA perverse sheaves on $\cGr_X$. For any $i\in I$, let $\calA_{X,i}\in P(\cGr_X)^{\rm ULA}$ such that $\calA_{X,i}$ are $\calL^+G_X$-equivariant for $i\geq 2$. We have the $\prod_{i\geq 2}\calL^+G_{X,i}$-equivariant ULA perverse sheaf $\boxtimes_{i\in I}\calA_{X,i}$ on $\prod_{i\in I}\cGr_{X,i}$.

\begin{lem}\label{twgloboxlem}
There is a unique ULA perverse sheaf $\boxtilde_{i\in I}\calA_{X,i}$ on $\tilde{\cGr}_I$ such that
there is a $q_I$-equivariant isomorphism\footnote{See Remark \ref{existencermk} below.}
\[q_I^*(\boxtilde_{i\in I}\calA_{X,i})\;\simeq\;p_I^*(\boxtimes_{i\in I}\calA_{X,i}),\]
where $q_I$-equivariant means with respect to the action on the $\calL^+G_I^o$-torsor $q_I:\tilde{\calL} G_I\to\tilde{\cGr}_I$. If $\calA_{X,1}$ is also $\calL^+G_X$-equivariant, then $\boxtilde_{i\in I}\calA_{X,i}$ is $\calL^+G_I$-equivariant
\end{lem}

\begin{rmk}\label{existencermk}
The ind-scheme $\tilde{\calL} G_I$ is not of ind-finite type. We explain how the pullback functors $p_I^*,q_I^*$ should be understood. Let $Y_1,\ldots,Y_k$ be $\calL^+G$-equivariant closed subschemes of $\cGr_X$ containing the supports of $\calA_1,\ldots,\calA_k$ . Choose $N>\!>0$ such that the action of $\calL^+G_X$ on each $Y_1,\ldots,Y_k$ factors over the smooth affine group scheme $H_N=\Res_{D^{(N)}/X}(G)$, where $D^{(N)}$ is the $N$-th infinitesimal neighbourhoud of the universal Cartier divisor $D$ over $X$. Let $K_N=\ker(\calL^+G_X\to H_N)$, and $Y=Y_1\times\ldots Y_k$. Then the left $K_N$-action on each $Y_i$ is trivial, and hence the restriction of the $p_I$-action resp. $q_I$-action on $p_I^{-1}(Y)$ to $\prod_{i\geq2}K_N$ agree. Let $h_N:p_I^{-1}(Y)\to Y_{N}$ be the resulting $\prod_{i\geq2}K_N$-torsor. By Lemma \ref{finitelem} below, we get a factorization
\[\begin{tikzpicture}[baseline=(current  bounding  box.center)]  
\matrix(a)[matrix of math nodes, 
row sep=2.2em, column sep=2.8em, 
text height=1.5ex, text depth=0.45ex] 
{ & p_I^{-1}(Y) &  \\ 
Y & Y_N & q_I(p_I^{-1}(Y)), \\}; 
\path[->](a-1-2) edge node[above] {$p_I$} (a-2-1) 
(a-1-2) edge node[right] {$h_N$} (a-2-2)
(a-1-2) edge node[above] {$q_I$} (a-2-3)
(a-2-2) edge node[below] {$p_{I,N}$} (a-2-1)
(a-2-2) edge node[below] {$q_{I,N}$} (a-2-3);  
\end{tikzpicture}\]
where $p_{I,N},q_{I,N}$ are $\prod_{i\geq2}H_N$-torsors. In particular, $Y_N$ is a separated scheme of finite type, and we can replace $p_I^*$ (resp. $q_I^*$) by $p_{I,N}^*$ (resp. $q_{I,N}^*$). 
\end{rmk}

\begin{proof}[Proof of Lemma \ref{twgloboxlem}.]
We use the notation from Remark \ref{existencermk} above. The sheaf $p_{I;N}^*(\boxtimes_{i\in I}\calA_{X,i})$ is $\prod_{i\geq2}H_N$-equivariant for the $q_{I,N}$-action. Using descent along smooth torsors (cf. Lemma \ref{stackylem} below), we get the perverse sheaf $\boxtilde_{i\in I}\calA_{X,i}$, which is ULA. Indeed, $p_{I;N}^*(\boxtimes_{i\in I}\calA_{X,i})$ is ULA, and the ULA property is local in the smooth topology. Since the diagram \eqref{globalconvdiagcopy} is $\calL^+G_I$-equivariant, the sheaf $\boxtilde_{i\in I}\calA_{X,i}$ is $\calL^+G_I$-equivariant, if $\calA_{X,1}$ is $\calL^+G_X$-equivariant. This proves the lemma.  
\end{proof}

Let $U_I$ be the open locus of pairwise distinct coordinates in $X^I$. There is a cartesian diagram
\[\begin{tikzpicture}[baseline=(current  bounding  box.center)] 
\matrix(a)[matrix of math nodes, 
row sep=1.5em, column sep=2em, 
text height=1.5ex, text depth=0.45ex] 
{\cGr_I & (\cGr_X^I)|_{U_I}\\ 
X^I & U_I.\\}; 
\path[->](a-1-1) edge (a-2-1); 
\path[->](a-1-2) edge node[above] {$j_I$} (a-1-1); 
\path[->](a-2-2) edge (a-2-1); 
\path[->](a-1-2) edge (a-2-2); 
\end{tikzpicture}\]

\begin{prop}\label{compareprop}
The complex $m_{I,*}(\boxtilde_{i\in I}\calA_{X,i})$ is a ULA perverse sheaf on $\cGr_I$, and there is a unique isomorphism of perverse sheaves
\[m_{I,*}(\boxtilde_{i\in I}\calA_{X,i})\;\simeq\; j_{I,!*}(\boxtimes_{i\in I}\calA_{X,i}|_{U_I}),\]
which is $\calL^+G_I$-equivariant, if $\calA_{X,1}$ is $\calL^+G_X$-equivariant.
\end{prop}
\begin{proof}
The sheaf $\boxtilde_{i\in I}\calA_{X,i}$ is by Lemma \ref{twgloboxlem} a ULA perverse sheaf on $\tilde{\cGr}_I$. Now the restriction of the global convolution morphism $m_I$ to the support of $\boxtilde_{i\in I}\calA_{X,i}$ is a proper morphism, and hence $m_{I,*}(\boxtilde_{i\in I}\calA_{X,i})$ is a ULA complex by Lemma \ref{ULAproperlem}. Then $m_{I,*}(\boxtilde_{i\in I}\calA_{X,i})\;\simeq\; j_{!*}((\boxtimes_{i\in I}\calA_{X,i})|_{U_I})$, as follows from Theorem \ref{mainthm} (i) and the formula $u_{!*}\circ v_{!*}\simeq(u\circ v)_{!*}$ for open immersions $V\stackrel{v}{\hookto} U\stackrel{u}{\hookto} T$, because $m_I|_{U_I}$ is an isomorphism. In particular, $m_{I,*}(\boxtilde_{i\in I}\calA_{X,i})$ is perverse. Since $m_I$ is $\calL^+G_I$-equivariant, it follows from proper base change that $m_{I,*}(\boxtilde_{i\in I}\calA_{X,i})$ is $\calL^+G_I$-equivariant, if $\calA_{X,1}$ is $\calL^+G_X$-equivariant. This proves the proposition. \\
\end{proof}

\subsection{The Symmetric Monoidal Structure}\label{smstr}
First we equip $P_{\calL^+G_X}(\cGr_X)^{\text{ULA}}$ with a symmetric monoidal structure $\vstar$ which allows us later to define a symmetric monoidal structure with respect to the usual convolution \eqref{convdiag} of $L^+G$-equivariant perverse sheaves on $\Gr_G$.

Fix $I$, and let $U_I$ be the open locus of pairwise distinct coordinates in $X^I$. Then the diagram
\begin{equation}\label{globalconvdiag}
\begin{tikzpicture}[baseline=(current  bounding  box.center)] 
\matrix(a)[matrix of math nodes, 
row sep=1.5em, column sep=2em, 
text height=1.5ex, text depth=0.45ex] 
{\cGr_X & \cGr_I & (\cGr_X^I)|_{U_I}\\ 
X & X^I & U_I.\\}; 
\path[->](a-1-1) edge node[above] {$i_I$} (a-1-2); 
\path[->](a-2-1) edge node[above] {$\text{diag}$} (a-2-2); 
\path[->](a-1-1) edge (a-2-1);
\path[->](a-1-2) edge (a-2-2); 
\path[->](a-1-3) edge node[above] {$j_I$} (a-1-2); 
\path[->](a-2-3) edge (a-2-2); 
\path[->](a-1-3) edge (a-2-3); 
\end{tikzpicture}
\end{equation}
is cartesian.

\begin{dfn}
Fix some total order on $I$. For every tuple $(\calA_{X,i})_{i\in I}$ with $\calA_{X,i}\in P(\cGr_X)^{\text{ULA}}$ for $i\in I$, the \emph{$I$-fold fusion product} $\vstar_{i\in I}\calA_{X,i}$ is the complex
 \[\vstar_{i\in I}\calA_{X,i}\defined i_I^*[-k+1]j_{I,!*}((\boxtimes_{i\in I}\calA_{X,i})|_{U_I})\;\in D_c^b(\cGr_X,\algQl),\]
where $k=|I|$.
\end{dfn}

Now let $\pi:I\to J$ be a surjection of finite index sets. For $j\in J$, let $I_j=\pi^{-1}(j)$, and denote by $U_\pi$ the open locus in $X^I$ such that the $I_j$-coordinates are pairwise distinct from the $I_{j'}$-coordinates for $j\not=j'$. Then the diagram  
\begin{equation}\label{assconvdiag}
\begin{tikzpicture}[baseline=(current  bounding  box.center)]  
\matrix(a)[matrix of math nodes, 
row sep=1.5em, column sep=2em, 
text height=1.5ex, text depth=0.45ex] 
{\cGr_J & \cGr_I & (\prod_j\cGr_{I_j})|_{U_\pi}\\ 
X^J & X^I & U_\pi,\\}; 
\path[->](a-1-1) edge node[above] {$i_\pi$} (a-1-2); 
\path[->](a-2-1) edge (a-2-2); 
\path[->](a-1-1) edge (a-2-1);
\path[->](a-1-2) edge (a-2-2); 
\path[->](a-1-3) edge node[above] {$j_\pi$} (a-1-2); 
\path[->](a-2-3) edge (a-2-2); 
\path[->](a-1-3) edge (a-2-3); 
\end{tikzpicture}
\end{equation}
is cartesian. The following theorem combined with Proposition \ref{compareprop} is the key to the symmetric monoidal structure:

\begin{thm}\label{mainsymthm}
Let $I$ be a finite index set, and let $\calA_{X,i}\in P_{\calL^+G_X}(\cGr_X)^{\text{ULA}}$ for $i\in I$. Let $\pi:I\to J$ be a surjection of finite index sets, and set $k_\pi=|I|-|J|$. \\
\emph{(i)} As complexes 
\[i_\pi^*[-k_\pi]j_{I,!*}((\boxtimes_{i\in I}\calA_{X,i})|_{U_I})\;\simeq\;i_\pi^![k_\pi]j_{I,!*}((\boxtimes_{i\in I}\calA_{X,i})|_{U_I}),\] 
and both are $\calL^+G_J$-equivariant ULA perverse sheaves on $\cGr_J$. Hence, $\vstar_{i\in I}\calA_{X,i}\in P_{\calL^+G_X}(\cGr_X)^{\text{ULA}}$. \\
\emph{(ii)} There is an associativity and a commutativity constraint for the fusion product such that there is a canonical isomorphism
\[\vstar_{i\in I}\calA_{X,i}\;\simeq\;\vstar_{j\in J}(\vstar_{i\in I_j}\calA_{X,i}),\]
where $I_j=\pi^{-1}(j)$ for $j\in J$. In particular, $(P_{\calL^+G_X}(\cGr_X)^{\text{ULA}},\vstar)$ is symmetric monoidal.
\end{thm}
\begin{proof}
Factor $\pi$ as a chain of surjective maps $I=I_1\to I_2\to\ldots\to I_{k_\pi}=J$ with $|I_{i+1}|=|I_i|+1$, and consider the corresponding chain of smooth Cartier divisors
\[X^J=X^{I_{k_\pi}}\longto\ldots\longto X^{I_2}\longto X^{I_1}=X^I.\] 
By Proposition \ref{compareprop}, the complex $j_{I,!*}((\boxtimes_{i\in I}\calA_{X,i})|_{U_I})$ is ULA. Then part (i) follows inductively from Theorem \ref{mainthm} (i) and (ii). This shows (i).\\
Let $\tau:I\to I$ be a bijection. Then $\tau$ acts on $X^I$ by permutation of coordinates, and diagram \eqref{globalconvdiag} becomes equivariant for this action. Then
\[\tau^*j_{I,!*}((\boxtimes_{i\in I}\calA_{X,i})|_{U_I})\;\simeq\; j_{I,!*}((\boxtimes_{i\in I}\calA_{X,\tau^{-1}(i)})|_{U_I}).\]
Since the action on $\text{diag}(X)\subset X^I$ is trivial, we obtain 
\[i_I^*j_{I,!*}((\boxtimes_{i\in I}\calA_{X,i})|_{U_I})\;\simeq\;i_I^*\tau^*j_{I,!*}((\boxtimes_{i\in I}\calA_{X,i})|_{U_I})\;\simeq\;i_I^* j_{I,!*}((\boxtimes_{i\in I}\calA_{X,\tau^{-1}(i)})|_{U_I}),\]
and hence $\vstar_{i\in I}\calA_{X,i}\simeq\vstar_{i\in I}\calA_{X,\tau^{-1}(i)}$. It remains to give the isomorphism defining the symmetric monoidal structure. Since $j_I=j_\pi\circ\prod_jj_{I_j}$, diagram \eqref{assconvdiag} gives
\[(j_{I,!*}((\boxtimes_{i\in I}\calA_{X,i})|_{U_I}))|_{U_\pi}\;\simeq\;\boxtimes_{j\in J}j_{I_j,!*}((\boxtimes_{i\in I_j}\calA_{X,i})|_{U_{I_j}}).\]
Applying $(i_\pi|_{U_\pi})^*[k_\pi]$ and using that $U_\pi\cap X^J=U_J$, we obtain
\[(i_\pi^*[k_\pi]j_{I,!*}((\boxtimes_{i\in I}\calA_{X,i})|_{U_I}))|_{U_J}\;\simeq\;\boxtimes_{j\in J}(\vstar_{i\in I_j}\calA_{X,i}).\]
But by (i), the perverse sheaf $i_\pi^*[k_\pi]j_{I,!*}((\boxtimes_{i\in I}\calA_{X,i})|_{U_I})$ is ULA, thus
\[i_\pi^*[k_\pi]j_{I,!*}((\boxtimes_{i\in I}\calA_{X,i})|_{U_I})\;\simeq\;j_{J,!*}((\boxtimes_{j\in J}(\vstar_{i\in I_j}\calA_{X,i}))|_{U_J}),\]
and restriction along the diagonal in $X^J$ gives the isomorphism $\vstar_{i\in I}\calA_{X,i}\;\simeq\;\vstar_{j\in J}(\vstar_{i\in I_j}\calA_{X,i})$. This proves (ii).
\end{proof} 

\begin{ex} Let $G=\{e\}$ be the trivial group. Then $\cGr_X=X$. Let $\text{Loc}(X)$ be the category of $\ell$-adic local systems on $X$. Using Corollary \ref{locsyscor}, we obtain an equivalence of symmetric monoidal categories
\[\calH^0\circ[-1]: (P(X)^{\text{ULA}},\vstar)\stackrel{\simeq}{\longto} (\text{Loc}(X),\otimes),\]
where $\text{Loc}(X)$ is endowed with the usual symmetric monoidal structure with respect to the tensor product $\otimes$. 
\end{ex}

\begin{cor}\label{fiberfun}
Let $D_c^b(X,\algQl)^{\rm ULA}$ be the category of ULA complexes on $X$. Denote by $f:\cGr_X\to X$ the structure morphism. Then the functor
\[f_*[-1]:(P(\cGr_X)^{\rm ULA},\vstar)\;\longto\;(D_c^b(X,\algQl),\otimes)\]
is symmetric monoidal.
\end{cor}
\begin{proof}
If $\calA_X\in P(\cGr_X)^{\text{ULA}}$, then $f_*\calA_X\in D_c^b(X,\algQl)^{\rm ULA}$ by Lemma \ref{ULAproperlem} and the ind-properness of $f$. Now apply $f_*$ to the isomorphism in Theorem \ref{mainsymthm} (ii) defining the symmetric monoidal structure on $P(\cGr_X)^{\rm ULA}$. Then by proper base change and going backwards through the arguments in the proof of Theorem \ref{mainsymthm} (ii), we get that $f_*[-1]$ is symmetric monoidal.
\end{proof}

\begin{cor}\label{fullsubcat}
Let $X=\bbA^1_F$. Let $p:\cGr_X\to \Gr_G$ be the projection, cf. Remark \ref{pullbackULAlem}. \\
\emph{(i)} The functor
\[p^*[1]:P_{L^+G}(\Gr_G)\longto P_{\calL^+G_X}(\cGr_X)^{\rm ULA}\]
embeds $P_{L^+G}(\Gr_G)$ as a full subcategory and is an equivalence of categories with the subcategory of $\bbG_a$-equivariant objects in $P_{\calL^+G_X}(\cGr_X)^{\rm ULA}$.\\ 
\emph{(ii)} For every $I$ and $\calA_i\in P_{L^+G}(\Gr_G)$, $i\in I$, there is a canonical $\calL^+G_X$-equivariant isomorphism 
\[p^*[1](\star_{i\in I}\calA_i)\;\simeq\;\vstar_{i\in I}(p^*[1]\calA_i),\]
where the product is taken with respect to some total order on $I$.
\end{cor}
\begin{proof}
Under the simply transitive action of $\bbG_a$ on $X$, the isomorphism $\cGr_X\simeq\Gr_G\times X$ is compatible with the action of $L^+G$ under the zero section $L^+G\hookto \calL^+G_X$. By Lemma \ref{pullbackULAlem}, the complex $p^*[1]\calA$ is a ULA perverse sheaf on $\cGr_X$. It is obvious that the functor $p^*[1]$ is fully faithful. Denote by $i_0:\Gr_G\to\cGr_X$ the zero section. If $\calA_X$ on $\cGr_X$ is $\bbG_a$-equivariant, then $\calA_X\simeq p^*[1]i_0^*[-1]\calA_X$. This proves (i).\\
By Remark \ref{usualconvdiag}, the fiber over $\text{diag}(\{0\})\in X^I(F)$ of \eqref{globalconvdiagcopy} is the usual convolution diagram \eqref{convdiag}. Hence, by proper base change,
\[i_0^*[-1](\vstar_{i\in I}p^*[1]\calA_i)\;\simeq\; \star_{i\in I}i_0^*[-1]p^*[1]\calA_i\;\simeq\;\star_{i\in I}\calA_i.\]
Since $\vstar_{i\in I}p^*[1]\calA_i$ is $\bbG_a$-equivariant, this proves (ii).
\end{proof}

Now we are prepared for the proof of Theorem \ref{monoidalthm}.

\begin{proof}[Proof of Theorem \ref{monoidalthm}.]
Let $X=\bbA^1_F$. For every $\calA_1,\calA_2\in P(\Gr_G)$ with $\calA_2$ being $L^+G$-equivariant, we have to prove that $\calA_1\star\calA_2\in P(\Gr_G)$. By Theorem \ref{mainsymthm} (i), the $\vstar$-convolution is perverse. Then the perversity of $\calA_1\star\calA_2$ follows from Corollary \ref{fullsubcat} (ii). Again by Corollary \ref{fullsubcat} (ii), the convolution $\calA_1\star\calA_2$ is $L^+G$-equivariant, if $\calA_1$ is $L^+G$-equivariant. This proves (i). \\   
We have to equip $(P_{L^+G}(\Gr_G),\star)$ with a symmetric monoidal structure. By Corollary \ref{fullsubcat}, the tuple $(P_{L^+G}(\Gr_G),\star)$ is a full subcategory of $(P_{\calL^+G_X}(\cGr_X)^{\rm ULA},\vstar)$, and the latter is symmetric monoidal by Theorem \ref{mainsymthm} (ii), hence so is $(P_{L^+G}(\Gr_G),\star)$. Since taking cohomology is only graded commutative, we need to modify the commutativity constraint of $(P_{L^+G}(\Gr_G),\star)$ by a sign as follows. Let $\sF$ be a separable closure of $F$. The $L^+G_{\sF}$-orbits in one connected component of $\Gr_{G,\sF}$ are all either even or odd dimensional. Because the Galois action on $\Gr_{G,\sF}$ commutes with the $L^+G_{\sF}$-action, the connected components of $\Gr_G$ are divided into those of even or odd parity. Consider the corresponding $\bbZ/2$-grading on $P_{L^+G}(\Gr_G)$ given by the parity of the connected components of $\Gr_G$. Then we equip $(P_{L^+G}(\Gr_G),\star)$ with the super commutativity constraint with respect to this $\bbZ/2$-grading, i.e. if $\calA$ (resp. $\calB$) is an $L^+G$-equivariant perverse sheaf supported on a connected component $X_{\calA}$ (resp. $X_{\calB}$) of $\Gr_G$, then the modified commutativity constraint differs by the sign $(-1)^{p(X_{\calA})p(X_{\calB})}$, where $p(X)\in\bbZ/2$ denotes the parity of a connected component $X$ of $\Gr_G$. 

Now consider the global cohomology functor 
\[\om(\str)=\bigoplus_{i\in\bbZ}R^i\Ga(\Gr_{G,\sF},(\str)_{\sF})\!: P_{L^+G}(\Gr_G) \longto \vs_{\algQl}.\]
Let $f:\cGr_X\to X$ be the structure morphism. Then the diagram
\[\begin{tikzpicture}[baseline=(current  bounding  box.center)]  
\matrix(a)[matrix of math nodes, 
row sep=2.2em, column sep=3em, 
text height=1.5ex, text depth=0.45ex] 
{P_{\calL^+G_{X,\sF}}(\cGr_{X,\sF})^{\rm ULA} & D_c^b(X_{\sF},\algQl) \\ 
P_{L^+G}(\Gr_G) & \vs_{\algQl}\\}; 
\path[->](a-1-1) edge node[above] {$f_*[-1]$} (a-1-2)
(a-2-1) edge node[left] {$p^*[1]\circ(\str)_{\sF}$} (a-1-1)
(a-2-1) edge node[above] {$\om$} (a-2-2)
(a-1-2) edge node[right] {$\oplus_{i\in\bbZ}\calH^i\circ i_0^*$} (a-2-2);  
\end{tikzpicture}\]
is commutative up to natural isomorphism. Now if $\calA$ is a perverse sheaf supported on a connected component $X$ of $\Gr_G$, then by a theorem of Lusztig \cite[Theorem 11c]{Lu},
\[\hspace{1,5cm} R^i\Ga(\Gr_{G,\sF},\calA_{\sF})=0,\hspace{2cm} i\not\equiv p(X) \pmod 2 ,\]
where $p(X)\in\bbZ/2$ denotes the parity of $X$. Hence, Corollary \ref{fiberfun} shows that $\om$ is symmetric monoidal with respect to the super commutativity constraint on $P_{L^+G}(\Gr_G)$. To prove uniqueness of the symmetric monoidal structure, it is enough to prove that $\om$ is faithful, which follows from Lemma \ref{addexfaithlem} below. This proves (ii).
\end{proof}

\section{The Tannakian Structure}\label{tannakastr}
In this section we assume that $F=\sF$ is separably closed. Let $X^\vee_+$ be a set of representatives of the $L^+G$-orbits on $\Gr_G$. For $\mu\in X^\vee_+$ we denote by $\calO_\mu$ the corresponding $L^+G$-orbit, and by $\olO_\mu$ its reduced closure with open embeddding $j^\mu:\calO_\mu\hookto\olO_\mu$. We equip $X^\vee_+$ with the partial order defined as follows: for every $\la,\mu\in X^\vee_+$, we define $\la\leq\mu$ if and only if $\calO_\la\subset\olO_\mu$.

\begin{prop}\label{sesilem}
The category $P_{L^+G}(\Gr_G)$ is semisimple with simple objects the intersection complexes
\[\hspace{1cm}\IC_\mu=j^\mu_{!*}\algQl[\dim(\calO_\mu)], \hspace{1cm} \text{for} \;\;\mu\in X_+^\vee.\]
In particular, if ${^pj}_*^\mu$ (resp. ${^pj}_!^\mu$) denotes the perverse push forward (resp. perverse extension by zero), then $j_{!*}^\mu\simeq{^pj}_!^\mu\simeq{^pj}_*^\mu$.
\end{prop}

\begin{proof}
For any $\mu\in X^\vee_+$, the \'etale fundamental group $\pi_1^{\text{\'et}}(\calO_\mu)$ is trivial. Indeed, since $\olO_\mu\backslash\calO_\mu$ is of codimension at least $2$ in $\olO_\mu$, Grothendieck's purity theorem implies that $\pi_1^{\text{\'et}}(\calO_\mu)=\pi_1^{\text{\'et}}(\olO_\mu)$. The latter group is trivial by [SGA1, XI.1 Corollaire 1.2], because $\olO_\la$ is normal (cf. \cite{Fa}), projective and rational. This shows the claim.\\
Since by \cite[Lemme 2.3]{NP} the stabilizers of the $L^+G$-action are connected, any $L^+G$-equivariant irreducible local system supported on $\calO_\mu$ is isomorphic to the constant sheaf $\algQl$. Hence, the simple objects in $P_{L^+G}(\Gr_G)$ are the intersection complexes $\IC_\mu$ for $\mu\in X^\vee_+$.\\
To show semisimplicity of the Satake category, it is enough to prove
\[\Ext^1_{P(\Gr_G)}(\IC_\la, \IC_\mu)=\Hom_{D^b_c(\Gr_G)}(\IC_\la, \IC_\mu[1])\stackrel{!}{=}0.\]
We distinguish several cases:\smallskip \\
\emph{Case} (i): $\la=\mu$.\smallskip\\
Let $\calO_\mu\stackrel{j}{\to}\olO_\mu\stackrel{i}{\leftarrow}\olO_\mu\bslash\calO_\mu$, and consider the exact sequence of abelian groups
\begin{equation}\label{outergroups}
\Hom(\IC_\mu,i_!i^!\IC_\mu[1])\longto\Hom(\IC_\mu,\IC_\mu[1])\longto\Hom(\IC_\mu,j_*j^*\IC_\mu[1])
\end{equation}
associated to the distinguished triangle $i_!i^!\IC_\mu\to\IC_\mu\to j_*j^*\IC_\mu$. We show that the outer groups in \eqref{outergroups} are trivial. Indeed, the last group is trivial, since $j^*\IC_\mu=\algQl[\dim(\calO_\mu)]$ gives 
\[\Hom(\IC_\mu,j_*j^*\IC_\mu[1])=\Hom(j^*\IC_\mu,j^*\IC_\mu[1])=\Ext^1(\algQl,\algQl).\]
And $\Ext^1(\algQl,\algQl)=H^1_{\text{\'et}}(\calO_\mu, \algQl)=0$, because $\calO_\mu$ is simply connected. To show that the first group
\[\Hom(\IC_\mu,i_!i^!\IC_\mu[1])=\Hom(i^*\IC_\mu,i^!\IC_\mu[1])\]
is trivial, we prove that $i^!\IC_\mu[1]$ lives in perverse degrees $\geq1$. Or equivalently, the Verdier dual $D(i^!\IC_\mu)=i^*\IC_\mu$ lives in perverse degrees $\leq-2$. By a theorem of Lusztig \cite[Theorem 11c]{Lu}, $i^*\IC_\mu$ is concentrated in even perverse degrees, and the $0$th perverse cohomology vanishes, since $\IC_\mu$ is a middle perverse extension along $j^\mu$. This proves case (i).\smallskip \\
\emph{Case} (ii): $\la\neq\mu$ {\it and either} $\la\leq\mu$ {\it or} $\mu\leq\la$.\smallskip \\
If $\la\leq\mu$, let $i:\olO_\la\hookto\olO_\mu$ be the closed embedding. Then 
\[\Hom(i_*\IC_\la,\IC_\mu[1])=\Hom(\IC_\la,i^!\IC_\mu[1]),\]
and this vanishes, since $i^!\IC_\mu[1]$ lives in perverse degrees $\geq 1$ by case (i) above. If $\mu\leq\la$, let $i:\olO_\mu\hookto\olO_\la$ the closed embedding. Then
\[\Hom(\IC_\la,i_*\IC_\mu[1])=\Hom(i^*\IC_\la,\IC_\mu[1])\] 
vanishes, since $i^*\IC_\la$ lives in perverse degrees $\leq-2$ as before. This proves case (ii). \smallskip \\
\emph{Case} (iii): $\la\not\leq\mu$ {\it and} $\mu\not\leq\la$.\smallskip\\
We may assume that $\la$ and $\mu$ are contained in the same connected component of $\Gr_G$. Choose some $\nu\in X_+^\vee$ with $\la,\mu\leq\nu$. Consider the cartesian diagram 
\[\begin{tikzpicture} 
\matrix(a)[matrix of math nodes, 
row sep=2.2em, column sep=2.2em, 
text height=1.5ex, text depth=0.45ex] 
{\olO_\la\times_{\olO_\nu}\olO_\mu & \olO_\mu \\ 
\olO_\la & \olO_\nu.\\}; 
\path[right hook->](a-1-1) edge node[above] {$\iota_1$} (a-1-2); 
\path[right hook->](a-2-1) edge node[above] {$i_1$} (a-2-2); 
\path[left hook->](a-1-1) edge node[right] {$\iota_2$} (a-2-1);
\path[right hook->](a-1-2) edge node[right] {$i_2$} (a-2-2); 
\end{tikzpicture}\]
Then adjunction gives
\begin{equation}\label{vanlanotmu}
\Hom(i_{1,*}\IC_\la,i_{2,*}\IC_\mu[1])=\Hom(i_2^*i_{1,*}\IC_\la,\IC_\mu[1]),
\end{equation}
and $i_2^*i_{1,*}\IC_\la\simeq\iota_{1,*}\iota_2^*\IC_\la$ by proper base change. Hence \eqref{vanlanotmu} equals $\Hom(\iota_2^*\IC_\la,\iota_1^!\IC_\mu[1])$ which vanishes. This proves case (iii), hence the proposition.
\end{proof}

The affine group scheme $L^+\bbG_m$ acts on $\Gr_G$ as follows. For $x\in L^+\bbG_m(R)$, denote by $v_x$ the automorphism of $\Spec(R\pot{t})$ induced by multiplication with $x$. If $\calF$ is a $G$-torsor over $\Spec(R\pot{t})$, we denote by $v_x^*\calF$ the pullback of $\calF$ along $v_x$. Let $(\calF,\beta)\in\Gr_G(R)$. Then the action of $L^+\bbG_m$ on $\Gr_G$ is given by
\[(\calF,\beta)\;\longmapsto\; (v_{x^{-1}}^*\calF,v_{x^{-1}}^*\beta),\]
and is called the \emph{Virasoro action}.  

Note that every $L^+G$-orbit in $\Gr_G$ is stable under $L^+\bbG_m$. The semidirect product $L^+G\rtimes L^+\bbG_m$ acts on $\Gr_G$, and the action on each orbit factors through a smooth connected affine group scheme. Hence, we may consider the category $P_{L^+G\rtimes L^+\bbG_m}(\Gr_G)$ of $L^+G\rtimes L^+\bbG_m$-equivariant perverse sheaves on $\Gr_G$. 

\begin{cor}\label{forgetcor}
The forgetful functor 
\[P_{L^+G\rtimes L^+\bbG_m}(\Gr_G)\longto P_{L^+G}(\Gr_G)\]
is an equivalence of categories. In particular, the category $P_{L^+G}(\Gr_G)$ does not depend on the choice of the parameter $t$.
\end{cor}

\begin{proof}
By Proposition \ref{sesilem} above, every $L^+G$-equivariant perverse sheaf is a direct sum of intersection complexes, and these are $L^+\bbG_m$-equivariant.
\end{proof}

\begin{rmk}
If $X=\bbA^1_F$ is the base curve, then the global affine Grassmannian $\cGr_X$ splits as $\cGr_X\simeq\Gr_G\times X$. Corollary \ref{forgetcor} shows that we can work over an arbitrary curve $X$ as follows. Let $\calX$ be the functor on the category of $F$-algebras $R$ parametrizing tuples $(x,s)$ with 
\begin{align*}
\begin{cases}
x\in X(R)\;\text{is a point};\\
s\;\text{is a continuous isomorphism of}\;R\text{-modules}\;\hat{\calO}_{X_R,x}\stackrel{\simeq}{\to}R\pot{t}, 
\end{cases}
\end{align*}
where $\hat{\calO}_{X_R,x}$ is the completion of the $R$-module $\calO_{X_R,x}$ along the maximal ideal $\mathfrak{m}_x$ at $x$. The affine group scheme $L^+\bbG_m$ operates from left on $\calX$ by $(g,(x,s))\mapsto (x,gs)$. The projection $p:\calX\to X$,$(x,s)\mapsto x$ gives $\calX$ the structure of a $L^+\bbG_m$-torsor. Then $\cGr_X\simeq\Gr_G\times^{L^+\bbG_m}\calX$, and we get a diagram of $L^+\bbG_m$-torsors
\begin{equation*}
\begin{tikzpicture}[baseline=(current  bounding  box.center)]  
\matrix(a)[matrix of math nodes, 
row sep=1.5em, column sep=2em, 
text height=1.5ex, text depth=0.45ex] 
{& \Gr_G\times\calX & \\  
\phantom{hallo}\Gr_G\times X &  & \phantom{X\times}\cGr_X.\phantom{hallo}\\}; 
\path[->](a-1-2) edge node[above] {$p$} (a-2-1)
(a-1-2) edge node[above] {$q$} (a-2-3);
\end{tikzpicture}
\end{equation*}
For any $\calA\in P_{L^+G}(\Gr_G)$, the perverse sheaf $\calA\boxtimes\algQl[1]$ on $\Gr_G\times X$ is $L^+\bbG_m$-equivariant by Corollary \ref{forgetcor}. Hence, $p^*(\calA\boxtimes\algQl[1])$ descends along $q$ to a perverse sheaf $\calA\boxtilde\algQl[1]$ on $\cGr_X$.
\end{rmk}

We are going to define a fiber functor on $P_{L^+G}(\Gr_G)$. Denote by
\begin{equation}\label{fiberfundef}
\om(\str)=\bigoplus_{i\in\bbZ}R^i\Gamma(\Gr_G,\str): P_{L^+G}(\Gr_G)\;\to\;\vs_{\algQl}
\end{equation}
the cohomology functor with values in the category of finite dimensional $\algQl$-vector spaces.

\begin{lem}\label{addexfaithlem}
The functor $\om:P_{L^+G}(\Gr_G)\to\vs_{\algQl}$ is additive, exact and faithful.
\end{lem}

\begin{proof}
Additivity is immediate. Exactness follows from Proposition \ref{sesilem}, since every exact sequence splits, and $\om$ is additive. To show faithfulness, it is enough, again by Proposition \ref{sesilem}, to show that the intersection cohomology of the Schubert varieties is non-zero. Indeed, we claim that the intersection cohomology of any projective variety $T$ is non-zero. Embedding $T$ into projective space and projecting down on hyperplanes, we obtain a generically finite morphism $\pi:T\to\bbP^n$. Using the decomposition theorem, we see that the intersection complex of $\bbP^n$ appears as a direct summand in $\pi_*\IC_T$. Hence, the intersection cohomology of $T$ is non-zero. This proves the lemma.
\end{proof}

\begin{cor}\label{tannakacor}
The tuple $(P_{L^+G}(\Gr_G),\star)$ is a neutralized Tannakian category with fiber functor $\om:P_{L^+G}(\Gr_G)\to\vs_{\algQl}$.
\end{cor}
\begin{proof}
We check the criterion in \cite[Prop. 1.20]{DM}: \\
The category $(P_{L^+G}(\Gr_G),\star)$ is abelian $\algQl$-linear (cf. Appendix \ref{pervapp} below) and by Theorem \ref{monoidalthm} (ii) above symmetric monoidal.
To prove that $\om$ is a fiber functor, we must show that $\om$ is an additive exact faithful tensor functor. Lemma \ref{addexfaithlem} shows that $\om$ is additive exact and faithful, and Theorem \ref{monoidalthm} (ii) shows that $\om$ is symmetric monoidal.\\
It remains to show that $(P_{L^+G}(\Gr_G),\star)$ has a unit object and that any one dimensional object has an inverse. The unit object is the constant sheaf $\IC_0=\algQl$ concentrated in the base point $e_0$. We have $\End(\IC_0)=\algQl$, and $\dim(\om(\IC_0))=1$. Now, let $\calA\in P_{L^+G}(\Gr_G)$ with $\dim(\om(A))=1$. Then $\calA$ is supported on a $L^+G$-invariant closed point $z_0\in\Gr_G$. There exists $z$ in the center of $LG$ such that $z\cdot z_0=e_0$ is the basepoint. If $z'_0=z\cdot e_0$, then the intersection cohomology complex $\calA'$ supported on $z'_0$ satisfies $\calA\star\calA'=\IC_0$. This shows the corollary.
\end{proof}

\section{The Geometric Satake Equivalence}\label{tannakaeq}
In this section we assume that $F=\sF$ is separably closed. Denote by $H=\Aut^\star(\om)$ the affine $\algQl$-group scheme of tensor automorphisms defined by Corollary \ref{tannakacor}. 

\begin{Thm}\label{sateq}
The group scheme $H$ is a connected reductive group over $\algQl$ which is dual to $G$ in the sense of Langlands, i.e. if we denote by $\hat{G}$ the Langlands dual group with respect to some pinning of $G$, then there exists an isomorphism $H\simeq\hat{G}$ determined uniquely up to inner automorphisms. 
\end{Thm}

We fix some notation. Let $T$ a maximal split torus of $G$ and $B$ a Borel subgroup containing $T$ with unipotent radical $U$. We denote by $\langle\str,\str\rangle$ the natural pairing between $X=\Hom(T,\bbG_m)$ and $X^\vee=\Hom(\bbG_m,T)$. Let $R\subset X$ be the root system associated to $(G,T)$, and $R_+$ be the set of positive roots corresponding to $B$. Let $R^\vee\subset X^\vee$ the dual root system with the bijection $R\to R^\vee$, $\al\mapsto\al^\vee$. Denote by $R^\vee_+$ the set of positive coroots. Let $W$ the Weyl group of $(G,T)$. Consider the half sum of all positive roots
\[\rho=\frac{1}{2}\sum_{\al\in R_+}\al.\]
Let $Q^\vee$ (resp. $Q_+^\vee$) the subgroup (resp. submonoid) of $X^\vee$ generated by $R^\vee$ (resp. $R_+^\vee$). We denote by 
\[X_+^\vee=\{\mu\in X^\vee \;|\; \langle\al,\mu\rangle\geq0, \;\forall\al\in R_+\}\] 
the cone of dominant cocharacters with the partial order on $X^\vee$ defined as follows: $\la\leq\mu$ if and only if $\mu-\la\in Q_+^\vee$. 

Note that $(X^\vee_+,\leq)$ identifies with the partially ordered set of orbit representatives in Section \ref{tannakastr} as follows: for every $\mu\in X_+^\vee$, let $t^\mu$ the corresponding element in $LT(F)$, and denote by $e_0\in\Gr_G$ the base point. Then $\mu \mapsto t^\mu\cdot e_0$ gives the bijection of partial ordered sets, i.e. the orbit closures satisfy
\[\hspace{3cm} \olO_\mu \;=\; \coprod_{\la\leq\mu}\calO_{\la},\hspace{1.2cm}\text{(Cartan stratification)}\]
where $\calO_\la$ denotes the $L^+G$-orbit of $t^\la\cdot e_0$ (cf. \cite[\S 2]{NP}).

For every $\nu\in X^\vee$, consider the $LU$-orbit $S_\nu=LU\cdot t^\nu e_0$ inside $\Gr_G$ (cf. \cite[\S 3]{NP}). Then $S_\nu$ is a locally closed ind-subscheme of $\Gr_G$, and for every $\mu\in X^\vee_+$, there is a locally closed stratification 
\[\hspace{3cm}\olO_\mu \;=\; \coprod_{\nu\in X^\vee}S_\nu\cap\olO_{\mu}.\hspace{1.2cm}\text{(Iwasawa stratification)}\]

For $\mu\in X^\vee_+$, let
\[\Omega(\mu)\defined\{\nu\in X^\vee \;|\; w\nu\leq\mu, \;\forall w\in W\}.\]

\begin{prop}
For every $\nu\in X^\vee$ and $\mu\in X^\vee_+$ the stratum $S_\nu\cap\olO_\mu$ is non-empty if and only if $\nu\in\Omega(\mu)$, and in this case it is pure of dimension $\langle\rho,\mu+\nu\rangle$.
\end{prop}
\begin{proof}
The schemes $G$, $B$, $T$ and all the associated data are already defined over a finitely generated $\bbZ$-algebra. By generic flatness, we reduce to the case where $F=\bbF_q$ is a finite field. The proposition is proven in \cite[Proof of Lemma 2.17.4]{GHKR}, which relies on \cite[Theorem 3.1]{NP}. 
\end{proof}

For every sequence $\mu_\blt=(\mu_1,\ldots,\mu_k)$ of dominant cocharacters, consider the projective variety over $F$
\[\olO_{\mu_\blt}\defined p^{-1}(\olO_{\mu_1})\times^{L^+G}\ldots\times^{L^+G} p^{-1}(\olO_{\mu_{k-1}})\times^{L^+G}\olO_{\mu_k},\]
inside $LG\times^{L^+G}\ldots\times^{L^+G}\Gr_G$, where $p:LG\to \Gr_G$ denotes the quotient map. The quotient exists, by the ind-properness of $\Gr_G$ and Lemma \ref{finitelem} below. 

Now let $|\mu_\blt|=\mu_1+\ldots+\mu_k$. Then the restriction $m_{\mu_\blt}=m|_{\olO_{\mu_\blt}}$ of the $k$-fold convolution morphism factors as
\[m_{\mu_\blt}:\olO_{\mu_\blt}\;\longto\;\olO_{|\mu_\blt|},\] 
and is an isomorphism over $\calO_{|\mu_\blt|}\subset\olO_{|\mu_\blt|}$.

\begin{cor}\label{semismallcor}
For every $\la\in X_+^\vee$ with $\la\leq|\mu_\blt|$ and $x\in\calO_\la(F)$, one has
\[\dim( m_{\mu_\blt}^{-1}(x))\;\leq\;\langle\rho,|\mu_\blt|-\la\rangle,\]
i.e. the convolution morphism is semismall.
\end{cor}
\begin{proof}
The proof of \cite[Lemme 9.3]{NP} carries over word by word, and we obtain that $\dim( m_{\mu_\blt}^{-1}(\calO_\la))\;\leq\;\langle\rho,|\mu_\blt|+\la\rangle$. Since $m_{\mu_\blt}$ is $L^+G$-equivariant and $\dim(\calO_\la)=\langle2\rho,\la\rangle$, the corollary follows.  
\end{proof}

The convolution $\IC_{\mu_1}\star\ldots\star\IC_{\mu_n}$ is a $L^+G$-equivariant perverse sheaf, and by Proposition \ref{sesilem}, we can write
\begin{equation}\label{multdecom}
\IC_{\mu_1}\star\ldots\star\IC_{\mu_n}\;\simeq\;\bigoplus_{\la\leq|\mu_\blt|}V_{\mu_\blt}^\la\otimes\IC_\la,
\end{equation}
where $V_{\mu_\blt}^\la$ are finite dimensional $\algQl$-vector spaces.

\begin{lem}\label{multlem}
For every $\la\in X^\vee_+$ with $\la\leq|\mu_\blt|$ and $x\in\calO_\la(F)$, the vector space $V_{\mu_\blt}^\la$ has a canonical basis indexed by the irreducible components of $m_{\mu_\blt}^{-1}(x)$ of exact dimension $\langle\rho,|\mu_{\blt}|-\la\rangle$.
\end{lem}
\begin{proof}
We follow the argument in Haines \cite{Haines}. We claim that $\IC_{\mu_\blt}=\IC_{\mu_1}\tilde{\boxtimes}\ldots\tilde{\boxtimes}\,\IC_{\mu_k}$ is the intersection complex on $\olO_{\mu_\blt}$. Indeed, this can be checked locally in the smooth topology, and then easily follows from the definitions. Hence, the left hand side of \eqref{multdecom} is equal to $m_{\mu_\blt,*}(\IC_{\mu_\blt})$. If $d=-\dim(\calO_\la)$, then taking the $d$-th stalk cohomology at $x$ in \eqref{multdecom} gives by proper base change
\[R^d\Ga(m_{\mu_\blt}^{-1}(x),\IC_{\mu_\blt})\;\simeq\; V_{\mu_\blt}^\la.\]
Since $m_{\mu_\blt}:\olO_{\mu_\blt}\to\olO_{|\mu_\blt|}$ is semismall, the cohomology $R^d\Ga(m_{\mu_\blt}^{-1}(x),\IC_{\mu_\blt})$ admits by \cite[Lemma 3.2]{Haines} a canonical basis indexed by the top dimensional irreducible components. This proves the lemma.
\end{proof}

In the following, we consider $\olO_{\mu_\blt}$ as a closed projective subvariety of 
\[\olO_{\mu_1}\times\olO_{\mu_1+\mu_2}\times\ldots\times\olO_{\mu_1+\ldots+\mu_k},\]
via $(g_1,\ldots,g_k)\mapsto(g_1,g_1g_2,\ldots,g_1\ldots g_k)$. The lemma below is the geometric analogue of the PRV-conjecture.

\begin{lem}\label{geoprv}
For every $\la\in X_+^\vee$ of the form $\la=\nu_1+\ldots+\nu_k$ with $\nu_i\in W\mu_i$ for $i=1,\ldots,k$, the perverse sheaf $\IC_{\la}$ appears as a direct summand in $\IC_{\mu_1}\star\ldots\star\IC_{\mu_k}$. 
\end{lem}
\begin{proof}
Let $\nu=w(\nu_2+\ldots+\nu_k)$ be the unique dominant element in the $W$-orbit of $\nu_2+\ldots+\nu_k$. Then $\la=\nu_1+w^{-1}\nu$. Hence, by induction, we may assume $k=2$. By Lemma \ref{multlem}, it is enough to show that there exists $x\in\calO_{\la}(F)$ such that $m_{\mu_\blt}^{-1}(x)$ is of exact dimension $\langle\rho,|\mu_{\blt}|-\la\rangle$.

Let $w\in W$ such that $w\nu_1$ is dominant, and consider $w\la=w\nu_1+w\nu_2$. We denote by $S_{w\nu_\blt}\cap\olO_{\mu_\blt}$ the intersection inside $\olO_{\mu_1}\times\olO_{\mu_1+\mu_2}$
\[S_{w\nu_\blt}\cap\olO_{\mu_\blt}\defined (S_{w\nu_1}\times S_{w\nu_1+w\nu_2})\cap\olO_{\mu_\blt}.\]
The convolution is then given by projection on the second factor. By \cite[Lemme 9.1]{NP}, we have a canonical isomorphism
\[S_{w\nu_\blt}\cap\olO_{\mu_\blt}\;\simeq\; (S_{w\nu_1}\cap\olO_{\mu_1})\times(S_{w\nu_2}\cap\olO_{\mu_2}).\]
Let $y=(y_1,y_2)$ in $(S_{w\nu_\blt}\cap\olO_{\mu_\blt})(F)$. Since for $i=1,2$ the elements $w\nu_i$ are conjugate under $W$ to $\mu_i$, there exist by \cite[Lemme 5.2]{NP} elements $u_1,u_2\in L^+U(F)$ such that
\begin{align*}
& y_1=u_1t^{w\nu_1}\cdot e_0\\
& y_2=u_1t^{w\nu_1}u_2t^{w\nu_2}\cdot e_0.
\end{align*} 
The dominance of $w\nu_1$ implies $t^{w\nu_1}u_2t^{-w\nu_1}\in L^+U(F)$, and hence $Y=S_{w\nu_\blt}\cap\olO_{\mu_\blt}$ maps under the convolution morphism onto an open dense subset $Y'$ in $S_{w\la}\cap\calO_\la$. Denote by $h=m_{\mu_\blt}|_Y$ the restriction to $Y$. Both $Y$, $Y'$ are irreducible schemes (their reduced loci are isomorphic to affine space), thus by generic flatness, there exists $x\in Y'(F)$ such that
\[\dim(h^{-1}(x))=\dim(Y)-\dim(Y')=\langle\rho,|\mu_\blt|+w\la\rangle-\langle\rho,\la+w\la\rangle=\langle\rho,|\mu_\blt|-\la\rangle.\]
In particular, $\dim(m_{\mu_\blt}^{-1}(x))\geq\langle\rho,|\mu_\blt|-\la\rangle$, and hence equality by Corollary \ref{semismallcor}.
\end{proof}

For the proof of Theorem \ref{sateq}, we introduce a weaker partial order $\preceq$ on $X^\vee_+$ defined as follows: $\la\preceq\mu$ if and only if $\mu-\la\in\bbR_+Q_+^\vee$. Then $\la\leq\mu$ if and only if $\la\preceq\mu$ and their images in $X^\vee/Q^\vee$ coincide (cf. Lemma \ref{compareorder} below).

\begin{proof}[Proof of Theorem \ref{sateq}.]
We proceed in several steps: \smallskip \\
(1) \emph{The affine group scheme $H$ is of finite type over $\algQl$.} \smallskip \\
 By \cite[Proposition 2.20 (b)]{DM} this is equivalent to the existence of a tensor generator in $P_{L^+G}(\Gr_G)$. Now there exist $\mu_1,\ldots,\mu_k\in X^\vee_+$ which generate $X^\vee_+$ as semigroups. Then $\IC_{\mu_1}\oplus\ldots\oplus\IC_{\mu_k}$ is a tensor generator. \smallskip \\
(2) \emph{The affine group scheme $H$ is connected reductive.} \smallskip \\
 For every $\mu\in X^\vee_+$ and $k\in\bbN$, the sheaf $\IC_{k\mu}$ is a direct summand of $\IC_\mu^{\star k}$, hence the scheme $H$ is connected by \cite[Corollary 2.22]{DM}. By \cite[Proposition 2.23]{DM}, the connected algebraic group $H$ is reductive if and only if $P_{L^+G}(\Gr_G)$ is semisimple, and this is true by Proposition \ref{sesilem}. \smallskip \\
(3) \emph{The root datum of $H$ is dual to the root datum of $G$.} \smallskip \\
Let $(X',R',\Delta',X'^\vee,R'^\vee,\Delta'^\vee)$ the based root datum of $H$ constructed in Theorem \ref{catroot} below. By Lemma \ref{recon} below it is enough to show that we have an isomorphism of partially ordered semigroups
\begin{equation}\label{ordsemigrps}
(X^\vee_+,\leq) \;\stackrel{\simeq}{\longto}\; (X'_+,\leq').
\end{equation}
By Proposition \ref{sesilem}, the map $X^\vee_+\to X'_+$, $\mu\mapsto [\IC_\mu]$, where $[\IC_\mu]$ is the class of $\IC_\mu$ in $K_0^+P_{L^+G}(\Gr_G)$ is a bijection of sets. 

For every $\la,\mu\in X^\vee_+$, we claim that $\la\preceq\mu$ if and only if $[\IC_\la]\preceq'[\IC_\mu]$. Assume $\la\preceq\mu$, and choose a finite subset $F\subset X^\vee_+$ satisfying Proposition \ref{order} (iii). Let $\calA=\oplus_{\nu\in F}\IC_\nu$, and suppose $\IC_\chi$ is a direct summand of $\IC_\la^{\star k}$ for some $k\in\bbN$. In particular, $\chi\leq k\la$ and so $\chi\in WF+\sum_{i=1}^kW\mu$. By Lemma \ref{geoprv}, the sheaf $\IC_\chi$ is a direct summand of $\IC_\mu^{\star k}\star\calA$, which means $[\IC_\la]\preceq'[\IC_\mu]$. Conversely, assume $[\IC_\la]\preceq'[\IC_\mu]$. Using Proposition \ref{order} (iv) below, this translates, by looking at the support, into the following condition: there exists $\nu\in X^\vee_+$ such that $\olO_{k\la}\subset\olO_{k\mu +\nu}$ holds for infinitely many $k\in\bbN$. Equivalently, $k\la\leq k\mu + \nu$ for infinitely many $k\in\bbN$ which implies $\la\preceq\mu$. 
 
For every $\la,\mu\in X^\vee_+$, we claim that $[\IC_\la]+[\IC_\mu]=[\IC_{\la+\mu}]$ in $X'_+$: by the proof of Theorem \ref{catroot} below, $[\IC_\la]+[\IC_\mu]$ is the class of the maximal element appearing in $\IC_{\la}\star\IC_{\mu}$. Since the partial orders $\preceq$, $\preceq'$ agree, this is $[\IC_{\la+\mu}]$. 

It remains to show that the partial orders $\leq,\leq'$ agree. The identification $X^\vee_+=X'_+$ prolongs to $X^\vee=X'$. We claim that $Q^\vee_+=Q'_+$ under this identification and hence $Q^\vee=Q'$, which is enough by Lemma \ref{compareorder} below. Let $\al^\vee\in Q^\vee_+$ a simple coroot, and choose some $\mu\in X^\vee_+$ with $\langle\al,\mu\rangle=2$. Then $\mu+s_\al(\mu)=2\mu-\al^\vee$ is dominant, and hence $\IC_{2\mu-\al^\vee}$ appears by Lemma \ref{geoprv} as a direct summand in $\IC_\mu^{\star2}$. By Lemma \ref{posrootlem} this means $\al^\vee\in Q'_+$, and thus $Q^\vee_+\subset Q'_+$. Conversely, assume $\al'\in Q'_+$ has the property that there exists $\mu\in X'_+$ with $2\mu-\al'\in X_+'$ and $\IC_{2\mu-\al'}$ appears as a direct summand in $\IC_\mu^{\star2}$. Note that every element in $Q'_+$ is a sum of these elements. Then $2\mu-\al'\leq2\mu$, and hence $\al'\in Q^\vee_+$. This shows $Q'_+\subset Q^\vee_+$ and finishes the proof of \eqref{ordsemigrps}.   
\end{proof}

\section{Galois Descent}\label{galoisdescent}
Let $F$ be any field, and $G$ a connected reductive group defined over $F$. Fix a separable closure $\sF$, and let $\Ga_F={\rm Gal}(\sF/F)$ be the absolute Galois group. Let $\Rep_{\algQl}(\Ga_F)$ be the category of finite dimensional continuous $\ell$-adic Galois representations. For any object defined over $F$, we denote by a subscript $(\str)_{\sF}$ its base change to $\sF$. Consider the functor
\begin{align*}
\Om: P_{L^+G}(\Gr_G)&\;\longto\;\Rep_{\algQl}(\Ga_F) \\
\calA&\;\longmapsto\; \bigoplus_{i\in\bbZ}R^i\Ga(\Gr_{G,\sF},\calA_{\sF}).
\end{align*}
There are canonical isomorphisms of fpqc-sheaves $(LG)_{\sF}\simeq LG_{\sF}$, $(L^+G)_{\sF}\simeq L^+G_{\sF}$ and $\Gr_{G,\sF}\simeq\Gr_{G_{\sF}}$. Hence, $\Om\simeq\om\circ (\str)_{\sF}$, cf. \eqref{fiberfundef}. 

The absolute Galois group $\Ga_F$ operates on the Tannakian category $P_{L^+G_{\sF}}(\Gr_{G_{\sF}})$ by tensor equivalences compatible with the fiber functor $\om$. Hence, we may form the semidirect product $^LG=\Aut^{\star}(\om)(\algQl)\rtimes\Ga_F$ considered as a topological group as follows. The group $\Aut^{\star}(\om)(\algQl)$ is equipped with the $\ell$-adic topology, the Galois $\Ga_F$ group with the profinite topology and $^LG$ with the product topology. Let $\Rep_{\algQl}^c(^LG)$ be the full subcategory of the category finite dimensional continuous $\ell$-adic representations of $^LG$ such that the restriction to $\Aut^{\star}(\om)(\algQl)$ is algebraic.

\begin{thm}\label{GSEdescent}
The functor $\Om$ is an equivalence of abelian tensor categories
\begin{align*}
\Om:P_{L^+G}(\Gr_G)&\;\longto\;\Rep_{\algQl}^c(^LG) \\
\calA&\;\longmapsto\;\Om(\calA).
\end{align*}
\end{thm} 
 
The proof of Theorem \ref{GSEdescent} proceeds in several steps.

\begin{lem}\label{embedlem}
Let $H$ be an affine group scheme over a field $k$. Let $\Rep_{k}(H)$ be the category of algebraic representations of $H$, and let $\Rep_{k}(H(k))$ be the category of finite dimensional representations of the abstract group $H(k)$. Assume that $H$ is reduced and that $H(k)\subset H$ is dense. Then the functor 
\begin{align*}
\Psi:\Rep_{k}(H)&\;\longto\;\Rep_{k}(H(k)) \\
\rho&\;\longmapsto\;\rho(k)
\end{align*}
is a fully faithful embedding.
\end{lem}
\hfill\ensuremath{\Box}

We recall some facts on the Tannakian formalism from the appendix in \cite{RZ}. Let $(\calC,\otimes)$ be a neutralized Tannakian category over a
field $k$ with fiber functor $v$. We define a monoidal category $\Aut^\otimes(\calC,v)$ as follows. Objects are pairs $(\sigma,\al)$, where $\sigma:\calC\to\calC$ is a tensor automorphism and $\al:v\circ\sigma\to v$ is a natural isomorphism of tensor functors. Morphisms between $(\sigma,\al)$ and $(\sigma',\al')$ are natural tensor isomorphisms between $\sigma$ and $\sigma'$ that are compatible with $\al,\al'$ in an obvious way. The monoidal structure is given by compositions. Since $v$ is faithful, $\Aut^\otimes(\calC,v)$ is equivalent to a set, and in fact is a group. 

Let $H=\Aut^\otimes_{\calC}(v)$, the Tannakian group defined by $(\calC,v)$. There is a canonical action of $\Aut^\otimes(\calC,v)$ on $H$ by automorphisms as follows. Let $(\sigma,\al)$ be in $\Aut^\otimes(\calC,v)$ . Let $R$ be a $k$-algebra, and let $h:v_R\to v_R$ be a $R$-point of $H$. Then $(\sigma,\al)\cdot h$ is the following composition
\[v_R\stackrel{\al^{-1}}{\longto}v_R\circ\sigma\stackrel{h\circ\id}{\longto}v_R\circ\sigma\stackrel{\al}{\longto}v_R.\]

Let $\Ga$ be an abstract group. Then an action of $\Ga$ on $(\calC,v)$ is by definition a group homomorphism $\on{act}:\Ga\to\Aut^\otimes(\calC,v)$. 

Assume that $\Ga$ acts on $(\calC,v)$. Then we define $\calC^\Ga$, the category of $\Ga$-equivariant objects in $\calC$ as follows. Objects are $(X,\{c_\ga\}_{\ga\in\Ga})$, where $X$ is an object in $\calC$ and $c_\ga:\on{act}_\ga(X)\simeq X$ is an isomorphism, satisfying the natural cocycle condition, i.e. $c_{\ga'\ga}=c_{\ga'}\circ\on{act}_{\ga'}(c_\ga)$. The morphisms between $(X,\{c_\ga\}_{\ga\in\Ga})$ and $(X',\{c'_\ga\}_{\ga\in\Ga})$ are morphisms between $X$ and $X'$, compatible with $c_\ga,c'_\ga$ in an obvious way.

\begin{lem}\label{Tannaka}
Let $\Ga$ be a group acting on $(\calC,v)$. \smallskip \\
\emph{(i)} The category $\calC^\Ga$ is an abelian tensor category. \smallskip \\
\emph{(ii)} Assume that $H$ is reduced and that $k$ is algebraically closed. The functor $v$ is an equivalence of abelian tensor categories
\[\calC^\Ga\;\simeq\;\Rep_k^o(H(k)\rtimes\Ga)\]
where $\Rep_k^o(H(k)\rtimes\Ga)$ is the full subcategory of finite dimensional representations of the abstract group $H(k)\rtimes\Ga$ such that the restriction to $H(k)$ is algebraic.
\end{lem}

\begin{rmk} 
In fact, the category $\calC^\Ga$ is neutralized Tannakian with fiber functor $v$. If $\Ga$ is finite, then $\Aut^\otimes_{\calC^\Ga}(v)\simeq H\rtimes\Ga$. However, if $\Ga$ is not finite, then $\Aut^\otimes_{\calC^\Ga}(v)$ is in general not $H\rtimes\Ga$, where the latter is regarded as an affine group scheme. 
\end{rmk}

\begin{proof}[Proof of Lemma \ref{Tannaka}.]
The monoidal structure on $\calC^\Ga$ is defined as
\[(X,\{c_\ga\}_{\ga\in\Ga})\otimes (X',\{c'_\ga\}_{\ga\in\Ga})=(X'',\{c''_\ga\}_{\ga\in\Ga}),\]
where $X''=X\otimes X'$ and $c''_\ga:\on{act}_\ga(X'')\to X''$ is the composition
\[\on{act}_\ga(X\otimes X')\simeq\on{act}_\ga(X)\otimes\on{act}_\ga(X')\stackrel{c_\ga\otimes c'_\ga}{\longrightarrow} X\otimes X'.\] 
This gives $\calC^\Ga$ the structure of an abelian tensor category. \\
Now assume that $H$ is reduced and that $k$ is algebraically closed. It is enough to show that as tensor categories
\[\Psi:\Rep_k(H)^\Ga\;\stackrel{\simeq}{\longto}\;\Rep_k^o(H(k)\rtimes\Ga)\]
compatible with the forgetful functors. Let $((V,\rho),\{c_\ga\}_{\ga\in\Ga})\in\Rep_k(H)^\Ga$. Then we define $(V,\rho_\Ga)\in\Rep_k^o(H(k)\rtimes\Ga)$ by
\[(h,\ga)\;\longmapsto\;\rho(h)\circ\al_h(V)\circ v\circ c_\ga^{-1}\in\GL(V),\]
where $\alpha_h:v\circ\sigma_h\simeq v$ is induced by the action of $\Ga$ as above. Using the cocycle relation, one checks that this is indeed a representation. By Lemma \ref{embedlem}, the natural map
\[\Hom_H(\rho,\rho')\;\longto\;\Hom_{H(k)}(\rho(k),\rho'(k))\]
is bijective. Taking $\Ga$-invariants shows that the functor $\Psi$ is fully faithful. Essential surjectivity is obvious.
\end{proof}

Now we specialize to the case $(\calC,\otimes)=(P_{L^+G_{\sF}}(\Gr_{G,\sF}),\star)$ with fiber functor $v=\om$. Then the absolute Galois group $\Ga=\Ga_{\sF}$ acts on this Tannakian category (cf. Appendix \ref{contdescent}). 

\begin{proof}[Proof of Theorem \ref{GSEdescent}.]
\emph{The functor $\Om$ is fully faithful.} \smallskip \\
Let $P_{L^+G_{\sF}}(\Gr_{G,\sF})^{\Ga,c}$ be the full subcategory of $P_{L^+G_{\sF}}(\Gr_{G,\sF})^{\Ga}$ consisting of perverse sheaves together with a continuous descent datum (cf. Appendix \ref{contdescent}). By Lemma \ref{equdescent}, the functor $\calA\mapsto\calA_{\sF}$ is an equivalence of abelian categories $P_{L^+G}(\Gr_{G})\simeq P_{L^+G_{\sF}}(\Gr_{G,\sF})^{\Ga,c}$. Hence, we get a commutative diagram
\[\begin{tikzpicture} 
\matrix(a)[matrix of math nodes, 
row sep=2.2em, column sep=3em, 
text height=1.5ex, text depth=0.45ex] 
{ P_{L^+G_{\sF}}(\Gr_{G,\sF})^{\Ga} & \Rep_{\algQl}^o(^LG) \\ 
P_{L^+G}(\Gr_{G})& \Rep_{\algQl}^c(^LG),\\}; 
\path[->](a-1-1) edge node[above] {$\om$} (a-1-2); 
\path[->](a-2-1) edge node[above] {$\Om$} (a-2-2); 
\path[->](a-2-1) edge node[left] {$\calA\mapsto\calA_{\sF}$} (a-1-1);
\path[->](a-2-2) edge node[above] {} (a-1-2); 
\end{tikzpicture}\]
where $\om$ is an equivalence of categories by Lemma \ref{Tannaka} (ii), and where the vertical arrows are fully faithful. Hence, $\Om$ is fully faithful.\smallskip\\
\emph{The functor $\Om$ is essentially surjective.}\smallskip\\
Let $\rho$ be in $\Rep_{\algQl}^c(^LG)$. Without loss of generality, we assume that $\rho$ is indecomposable. Let $H=\Aut^*(\om)$. By Proposition \ref{sesilem}, the restriction $\rho|_H$ is semisimple. Denote by $A$ the set of isotypic components of $\rho|_H$. Then $\Ga_F$ operates transitively on $A$, and for every $a\in A$ its stabilizer in $\Ga_F$ is the absolute Galois group $\Ga_E$ for some finite separable extension $E/F$. By Galois descent along finite extensions, we may assume that $E=F$, and hence that $\rho|_H$ has only one isotypic component. Let $\rho_0$ be the simple representation occuring in $\rho|_H$. Then $\Hom_H(\rho_0,\rho)$ is a continuous $\Ga$-representation, and the natural morphism 
\[\rho_0\otimes\Hom_H(\rho_0,\rho)\;\longto\;\rho\]
given by $v\otimes f\mapsto f(v)$ is an isomorphism of $^LG$-representations. Let $\IC_X$ be the simple perverse sheaf on $\Gr_{G,\sF}$ with $\om(\IC_X)\simeq\rho_0$. Since $\rho$ has only one isotypic component, the support $X=\supp(\IC_X)$ is $\Ga$-invariant, and hence defined over $F$. Denote by $V$ the local system on $\Spec(F)$ given by the $\Ga$-representation $\Hom_H(\rho_0,\rho)$. Then $\IC_X\otimes V$ is an object in $P_{L^+G}(\Gr_G)$ such that $\Om(\IC_X\otimes V)\simeq\rho_0\otimes \Hom_H(\rho_0,\rho)$. This proves the theorem.
\end{proof}

The proof of Theorem \ref{GSEdescent} also shows the following fact.

\begin{cor}\label{indecom}
Let $\calA\in P_{L^+G}(\Gr_G)$ indecomposable. Let $\{X_i\}_{i\in I}$ be the set of irreducible components of $\supp(\calA_{\sF})$. Denote by $E$ the minimal finite separable extension of $F$ such that $X_i$ is defined over $E$ for all $i\in I$. Then as perverse sheaves on $\Gr_E$
\[\calA_E\;\simeq\; \bigoplus_{i\in I}\IC_{X_i}\otimes V_i,\]
where $V_i$ are indecomposable local systems on $\Spec(E)$. 
 \end{cor}
 \hfill\ensuremath{\Box}

We briefly explain the connection to the full $L$-group. For more details see the appendix in \cite{RZ}. Let $\hat{G}$ be the reductive group over $\algQl$ dual to $G_{\sF}$ in the sense of Langlands, i.e. the root datum of $\hat{G}$ is dual to the root datum of $G_{\sF}$. There are two natural actions of $\Ga_F$ on $\hat{G}$ as follows. Up to the choice of a pinning $(\hat{G},\hat{B},\hat{T},\hat{X})$ of $\hat{G}$, we have an action $\act^{\on{alg}}$ via
\begin{equation}\label{pinned}
\act^{\on{alg}}:\Ga_F\to\on{Out}(G_{\sF})\simeq\on{Out}(\hat{G})\simeq\Aut(\hat{G},\hat{B},\hat{T},\hat{X})\subset\Aut(\hat{G}),
\end{equation}
where $\on{Out}(\str)$ denotes the outer automorphisms. On the other hand, we have an action $\act^{\on{geo}}:\Ga_{\sF}\to\Aut(\hat{G})$ via the Tannakian equivalence from Theorem \ref{sateq}. The relation between $\act^{\on{geo}}$ and $\act^{\on{alg}}$ is as follows.\\
Let $\on{cycl}:\Ga_F\to\bbZ_\ell^\times$ be the cyclotomic character of $\Ga_F$ defined by the action of $\Ga_F$ on the $\ell^\infty$-roots of unity of $\sF$. Let $\hat{G}_{\ad}$ be the adjoint group of $\hat{G}$. Let $\rho$ be the half sum of positive coroots of $\hat{G}$, which gives rise to a one-parameter group $\rho:\bbG_m\to \hat{G}_\ad$. We define a map
\[\chi:\Ga_F\stackrel{\on{cycl}}{\longto} \bbZ_\ell^\times\stackrel{\rho}{\longto} \hat{G}_{\ad}(\algQl),\]
which gives a map $\Ad_{\chi}:\Ga_F\to\Aut(\hat{G})$ to the inner automorphism of $\hat{G}$.

\begin{prop}[\cite{RZ} Proposition A.4]\label{comparison}
For all $\ga\in\Ga_F$,
\[\act^{\on{geo}}(\ga)\;=\;\act^{\on{alg}}(\ga)\circ \Ad_{\chi}(\ga).\]
\end{prop}
\hfill\ensuremath{\Box}

\begin{rmk} Proposition \ref{comparison} shows that $\act^{\on{geo}}$ only depends on the quasi-split form of $G$, since the same is true for $\act^{\on{alg}}$. In particular, the Satake category $P_{L^+G}(\Gr_G)$ only depends on the quasi-split form of $G$ whereas the ind-scheme $\Gr_G$ does depend on $G$.
\end{rmk}

Let $^LG^{\on{alg}}=\hat{G}(\algQl)\rtimes_{\act^{\on{alg}}}\Ga_F$ be the \emph{full $L$-group}. Set $^LG^{\on{geo}}=\hat{G}(\algQl)\rtimes_{\act^{\on{geo}}}\Ga_F$.

\begin{cor}[\cite{RZ} Corollary A.5]\label{geom-alg}
The map $(g,\ga)\mapsto (\Ad_{\chi(\ga^{-1})}(g),\ga)$ gives an isomorphism ${^L}G^{\on{alg}}\stackrel{\simeq}{\to} {^L}G^{\on{geo}}$.
\end{cor}
\hfill\ensuremath{\Box}

Combining Corollary \ref{geom-alg} with Theorem \ref{GSEdescent}, we obtain the following corollary.

\begin{cor}
There is an equivalence of abelian tensor categories
\[P_{L^+G}(\Gr_G)\;\simeq\;\Rep_{\algQl}^c(^LG^{\on{alg}}),\]
where $\Rep_{\algQl}^c(^LG^{\on{alg}})$ denotes the full subcategory of the category of finite dimensional continuous $\ell$-adic representations of $^LG^{\on{alg}}$ such that the restriction to $\hat{G}(\algQl)$ is algebraic.
\end{cor}
\hfill\ensuremath{\Box}

\begin{appendix}
\section{Perverse Sheaves}\label{pervapp}
For the construction of the category of $\ell$-adic perverse sheaves, we refer to the work of Y. Laszlo and M. Olsson \cite{LO}. In this appendix we explain our conventions on perverse sheaves on ind-schemes. 

Let $F$ be an arbitrary field. Fix a prime $\ell\not=\cha(F)$, and denote by $\bbQ_\ell$ the field of $\ell$-adic numbers with algebraic closure $\algQl$. For any separated scheme $T$ of finite type over $F$, we consider the bounded derived category $D_c^b(T,\algQl)$ of constructible $\ell$-adic sheaves on $T$. Let $P(T)$ be the abelian $\algQl$-linear full subcategory of $\ell$-adic perverse sheaves, i.e. the heart of the perverse $t$-structure on the triangulated category $D_c^b(T,\algQl)$. 

Now let $(T)_{i\in I}$ be an inductive system of separated schemes of finite type over $F$ with closed immersions as transition morphisms. A fpqc-sheaf $\calT$ on the category of $F$-algebras is called a \emph{strict ind-scheme of ind-finite type over $F$} if there is an isomorphism of fpqc-sheaves $\calT\simeq\varinjlim_iT_i$, for some system $(T)_{i\in I}$ as above. The inductive system $(T)_{i\in I}$ is called an \emph{ind-presentation of $\calT$}.

For $i\leq j$, push forward gives transition morphisms $D_c^b(T_i,\algQl)\to D_c^b(T_j,\algQl)$ which restrict to $P(T_i)\to P(T_j)$, because push forward along closed immersions is $t$-exact.   

\begin{dfn} 
Let $\calT$ be a strict ind-scheme of ind-finite type over $F$, and $(T_i)_{i\in I}$ be an ind-presentation.\\
(i) The \emph{bounded derived category of constructible $\ell$-adic complexes} $D_c^b(\calT,\algQl)$ on $\calT$ is the inductive limit
\[D_c^b(\calT,\algQl)\defined \varinjlim_iD_c^b(T_i,\algQl).\]
(ii) The \emph{category of $\ell$-adic perverse sheaves} $P(\calT)$ on $\calT$ is the inductive limit
\[P(\calT)\defined \varinjlim_iP(T_i).\] 
\end{dfn}

The definition is independent of the chosen ind-presentation of $\calT$. The category $D_c^b(\calT,\algQl)$ inherits a triangulation and a perverse $t$-structure from the $D_c^b(T_i,\algQl)$'s. The heart with respect to the perverse $t$-structure is the abelian $\algQl$-linear full subcategory $P(\calT)$.

If $f:\calT\to \calS$ is a morphism of strict ind-schemes of ind-finite type over $F$, we have the Grothendieck operations $f_*,f_!,f^*,f^!$, and the usual constructions carry over after the choice of ind-presentations. 

In Section \ref{smstr} we work with equivariant objects in the category of perverse sheaves. The context is as follows. Let $f:T\to S$ be a morphism of separated schemes of finite type, and let $H$ be a smooth affine group scheme over $S$ with geometrically connected fibers acting on $f:T\to S$. Then a perverse sheaf $\calA$ on $T$ is called \emph{$H$-equivariant} if there is an isomorphism in the derived category
\begin{equation}\label{equivariance}
\theta: a^*\calA \;\simeq\; p^*\calA,
\end{equation}
where $a:H\times_S T\to T$ (resp. $p:H\times_S T\to T$) is the action (resp. projection on the second factor). A few remarks are in order: if the isomorphism \eqref{equivariance} exists, then it can be rigidified such that $e_T^*\theta$ is the identity, where $e_T:T\to H\times_S T$ is the identity section. A rigidified isomorphism $\theta$ automatically satisfies the cocycle relation due to the fact that $H$ has geometrically connected fibers. 

The subcategory $P_H(T)$ of $P(T)$ of $H$-equivariant objects together with $H$-equivariant morphisms is called  the \emph{category of $H$-equivariant perverse sheaves} on $T$. 

\begin{lem}[\cite{LO} Remark 5.5]    \label{stackylem}
Consider the stack quotient $H\bslash T$, an Artin stack of finite type over $S$. Let $p:T\to H\bslash T$ be the quotient map of relative dimension $d=\dim(T/S)$. Then the pull back functor
\[p^*[d]:P(H\bslash T)\;\longto\;P_H(T),\]
is an equivalence of categories. In particular, $P_H(T)$ is abelian and $\algQl$-linear. 
\end{lem}
\hfill\ensuremath{\Box}

Now let $\calT$ be a strict ind-scheme of ind-finite type, and $f:\calT\to S$ a morphism to a separated scheme of finite type. Fix an ind-presentation $(T_i)_{i\in I}$ of $\calT$. Let $(H_i)_{i\in I}$ be an inverse system of smooth affine group scheme with geometrically connected fibers. Let $\calH=\varprojlim_iH_i$ be the inverse limit, an affine group scheme over $S$, because the transition morphism are affine. Assume that $\calH$ acts on $f:\calT\to S$ such that the action restricts to the inductive system $(f|_{T_i})_{i\in I}$. Assume that the $\calH$-action factors through $H_i$ on $f|_{T_i}$ for every $i\in I$.

\begin{dfn}
Let $f:\calT\to S$, $(T_i)_{i\in I}$ and $\calH$ as above. The \emph{category $P_{\calH}(\calT)$ of $\calH$-equivariant perverse sheaves on $\calT$} is the inductive limit
\[P_{\calH}(\calT)\defined\varinjlim_iP_{H_i}(T_i).\]
\end{dfn}

It follows from Lemma \ref{stackylem} that the category $P_{\calH}(\calT)$ is an abelian $\algQl$-linear category. The following lemma is used throughout the text.

\begin{lem}\label{finitelem}
Let $T\to S$ be a $\calH$-torsor, and let $Y$ be a $S$-scheme with $\calH$-action. Assume that the action of $\calH$ on $Y$ factors over $H_i$ for $i>\!>0$. Then there is a canonical isomorphism of fpqc-sheaves
\[T\times^{\calH}Y \;\overset{\simeq}{\longto}\; T^{(i)}\times^{H_i}Y,\]
where $T^{(i)}=T\times^{\calH}H_i$.
\end{lem}
\hfill\ensuremath{\Box}

\begin{rmk}
 In particular, if $T^{(i)}\times^{H_i}Y$ is representable of finite type, then is $T\times^{\calH}Y$ is representable of finite type.
\end{rmk}

\subsection{Galois Descent of Perverse Sheaves}\label{contdescent}
Fix a separable closure $\sF$ of $F$. Let $\Ga={\rm Gal}(\sF/F)$ be the absolute Galois group. For any complex of sheaves $\calA$ on $T$, we denote by $\calA_{\sF}$ its base  change to $T_{\sF}=T\otimes\sF$. We define the category of \emph{perverse sheaves with continuous $\Ga$-descent datum} $P(T_{\sF})^{\Ga,c}$ as follows. The objects are pairs $(\calA,\{c_\ga\}_{\ga\in\Ga})$, where $\calA\in P(T_{\sF})$ and $\{c_\ga\}_{\ga\in\Ga}$ is a family of isomorphisms
\[c_{\ga}:\ga_*\calA\;\stackrel{\simeq}{\longto}\;\calA,\]
satisfying the cocycle condition $c_{\ga'\ga}=c_{\ga}\circ \ga'_*(c_{\ga})$ such that the datum is continuous in the following sense. For every $i\in\bbZ$ and every locally closed subscheme $S\subset T$ such that the standard cohomology sheaf $\calH^i(\calA)|_S$ is a local system, and for every $U\to S$ \'etale, with $U$ separated quasi-compact, the induced $\ell$-adic Galois representation on the $U_{\sF}$-sections 
\[\Ga\;\longto\;\GL(\calH^i(\calA)(U_{\sF})),\]
is continuous. The morphisms in $P(T_{\sF})^{\Ga,c}$ are morphisms in $P(T_{\sF})$ compatible with the $c_\ga$'s. For every $\calA\in P(T)$, its pullback $\calA_{\sF}$ admits a canonical continuous descent datum. Hence, we get a functor
\begin{align*}
\Phi:P(T) &\;\longto\; P(T_{\sF})^{\Ga,c} \\
\calA &\;\longmapsto\; \calA_{\sF}.
\end{align*}

\begin{lem}[SGA 7, XIII, 1.1] \label{equdescent}
The functor $\Phi$ is an equivalence of categories.
\end{lem}
\hfill\ensuremath{\Box}

\section{Reconstruction of Root Data}\label{reconapp}
Let $G$ a split connected reductive group over an arbitrary field $k$. Denote by $\Rep_{G}$ the Tannakian category of algebraic representations of $G$. If $k$ is algebraically closed of characteristic $0$, then D. Kazhdan, M. Larsen and Y. Varshavsky \cite[Corollary 2.5]{KLV} show how to reconstruct the root datum of $G$ from the Grothendieck semiring $K_0^+[G]=K_0^+\Rep_{G}$. In fact, their construction works over arbitrary fields. This relies on the conjecture of Parthasarathy, Ranga-Rao and Varadarajan (PRV-conjecture) proven by S. Kumar \cite{Kumar} ($\cha(k)=0$) and O. Mathieu \cite{Mathieu} ($\cha(k)>0$).

\begin{Thm} \label{catroot}
The root datum of $G$ can be reconstructed from the Grothendieck semiring $K_0^+[G]$. 
\end{Thm}

This means, if $H$ is another split connected reductive group over $k$, and if $\varphi: K_0^+[H]\to K_0^+[G]$ is an isomorphism of Grothendieck semirings, then there exists an isomorphism of group schemes $\phi:H\to G$ determined uniquely up to inner automorphism such that $\phi=K_0^+[\varphi]$. 

Let $T$ be a maximal split torus of $G$ and $B$ a Borel subgroup containing $T$. We denote by $\langle\str,\str\rangle$ the natural pairing between $X=\Hom(T,\bbG_m)$ and $X^\vee=\Hom(\bbG_m,T)$. Let $R\subset X$ be the root system associated to $(G,T)$, and $R_+$ be the set of positive roots corresponding to $B$. Let $R^\vee\subset X^\vee$ the dual root system with the bijection $R\to R^\vee$, $\al\mapsto\al^\vee$. Denote by $R^\vee_+$ the set of positive coroots. Let $W$ the Weyl group of $(G,T)$. Consider the half sum of all positive roots
\[\rho=\frac{1}{2}\sum_{\al\in R_+}\al.\]
Let $Q$ (resp. $Q_+$) the subgroup (resp. submonoid) of $X$ generated by $R$ (resp. $R_+$). We denote by 
\[X_+=\{\mu\in X \;|\; \langle\mu,\al\rangle\geq0, \;\forall\al\in R_+^{\vee}\}\] 
the cone of dominant characters. 

We consider partial orders $\leq$ and $\preceq$ on $X$ defined as follows. For $\la,\mu\in X$, we define  $\la\leq\mu$ if and only if $\mu-\la\in Q_+$, and we define $\la\preceq\mu$ if and only if $\mu-\la=\sum_{\al\in\Delta}x_\al\al$ with $x_\al\in\bbR_{\geq0}$. The latter order is weaker than the former order in the sense that $\la\leq\mu$ implies $\la\preceq\mu$, but in general not conversely. 

\begin{lem}[\cite{Ra}]\label{compareorder}
For every $\la,\mu\in X_+$, then $\la\leq\mu$ if and only if $\la\preceq\mu$ and the images of $\la,\mu$ in $X/Q$ agree.
\end{lem}

Let
\[\Dom_{\preceq \mu}=\{\nu\in X_+ \;|\; \nu\preceq\mu\}.\]

For a finite subset $F$ of the euclidean vector space $E=X\otimes\bbR$, we denote by $\Conv(F)$ its convex hull. 

\begin{prop}\label{order}
 For $\la,\mu\in X_+$, the following conditions are equivalent: \smallskip \\
\emph{(i)} $\la\preceq\mu$ \smallskip \\
\emph{(ii)} $\Conv(W\la)\subset \Conv(W\mu)$ \smallskip \\
\emph{(iii)} There exists a finite subset $F\subset X_+$ such that for all $k\in\bbN$:
\[\Dom_{\preceq k\la} \;\;\subset\;\; WF + \sum_{i=1}^kW\mu\]
\emph{(iv)} There exists a representation $U$ such that for every $k\in\bbN$, every irreducible subquotient of $V_\la^{\otimes k}$ is a subquotient of $V_\mu^{\otimes k}\otimes U$.
\end{prop}
\begin{proof}
The equivalence of (i) and (ii) is well-known. The implication (ii)$\Rightarrow$(iii) follows from \cite[Lemma 2.4]{KLV}. Assume (iii), we show that (iv) holds: let $U=\oplus_{\nu\in F}V_\nu$, and suppose $V_\chi$ is a irreducible subquotient of $V_\la^{\otimes k}$, in particular $\chi\leq k\la$. By (iii), $\chi$ has the form $w\nu + \sum_{i=1}^kw_i\mu$ with $w, w_1,\ldots,w_k\in W$ and $\nu\in F$. Using the PRV-conjecture \cite[Theorem 4.3.2]{BrKu}, we conclude that $V_\chi$ is a subquotient of $V_\mu^{\otimes k} \otimes V_\nu$, hence also of $V_\mu^{\otimes k} \otimes U$. This shows (iv). The implication (iv)$\Rightarrow$(i) is shown in \cite[Proposition 2.2]{KLV}.   
\end{proof}

For $\mu\in X_+$, let $v_\mu$ be the corresponding element in $K_0^+[G]$. Let $\calQ_+\subset X$ be the semigroup generated by the set
\[\{\al\in X\;|\;\exists\, \mu\in X_+:\;2\mu-\al\in X_+\;\text{and}\;v_\mu^2-v_{2\mu-\al}\in K_0^+[G]\}.\]

\begin{lem}\label{posrootlem}
There is an equality of semigroups $\calQ_+=Q_+$.
\end{lem}
\begin{proof}
It is obvious that $\calQ_+\subset Q_+$, and we show that $\calQ_+$ contains the simple roots. Let $\al$ be a simple root, and choose some $\mu\in X$ such that $\langle\mu,\al^\vee\rangle=2$. Then $2\mu-\al$ paired with any simple root is positive, and hence $\mu+s_\al(\mu)=2\mu-\al$ is dominant. By the $PRV$-conjecture \cite[Theorem 4.3.2]{BrKu}, the representation $V_{2\mu-\al}$ appears as an irreducible subquotient in $V_\mu^{\otimes2}$, i.e. $v_\mu^2-v_{2\mu-\al}\in K_0^+[G]$.
\end{proof}

The proof of Theorem \ref{catroot} goes along the lines of \cite[Corollary 2.5]{KLV}.
\begin{proof}[Proof of Theorem \ref{catroot}.]
By Lemma \ref{recon} below it is enough to construct the partially ordered semigroup $(X_+,\leq)$ of dominant weights.\\
The underlying set of dominant weights $X_+$ is the set of irreducible objects in $K_0^+[G]$. Then the partial order $\preceq$ on $X_+$ is characterized by Proposition \ref{order} as follows: for $\la,\mu\in X_+$, one has $\la\preceq\mu$ if and only if there exists a $u\in K_0^+[G]$ such that for all $k\in\bbN$ and $\nu\in X_+$,    
\[v_\la^k - v_\nu\in K_0^+[G] \;\;\; \Longrightarrow \;\;\; v_\mu^k\cdot u - v_\nu \in K_0^+[G].\]
The semigroup structure on $X_+$ is given by: for $\la,\mu\in X_+$, one has $\nu=\la+\mu$ if and only if $\nu$ is the unique dominant weight which is maximal (w.r.t. $\preceq$) with the property that $v_\la\cdot v_\mu-v_\nu\in K_0^+[G]$. \\
Now $X$ is the group completion of $X_+$, and by Lemma \ref{posrootlem} we can reconstruct $Q_+\subset X$. Then $Q$ is the group completion of $Q_+$, and by Lemma \ref{compareorder} we can reconstruct $\leq$. This shows that the root datum of $G$ can be reconstructed from $K_0^+[G]$. 

Now if $H$ is another split connected reductive group over $k$, and $\varphi:K_0^+[H]\to K_0^+[G]$ an isomorphism of Grothedieck semirings, then the argument above shows that there is an isomorphism of partially ordered semigroups
\begin{equation}\label{partiso}
(X_+^H,\leq^H)\longto (X_+^G,\leq^G)
\end{equation}
inducing $\varphi$ on Grothendieck semirings. By Lemma \ref{recon} below, the morphism \ref{partiso} prolongs to an isomorphism of the associated based root data. Hence, there exists an isomorphism of group schemes $\phi: H\to G$ inducing the isomorphism of based root data. In particular, $\varphi=K_0^+[\phi]$, and such an isomorphism $\phi$ is uniquely determined up to inner automorphism. This finishes the proof of Theorem \ref{catroot}.
\end{proof}

\begin{lem}\label{recon}
Let $\calB=(X,R,\Delta,X^\vee,R^\vee,\Delta^\vee)$ any based root datum. Denote by $(X_+,\leq)$ the partially ordered semigroup of dominant weights. Then the root datum $\calB$ can be reconstructed from $(X_+,\leq)$, i.e. if $\calB'=(X',R',\Delta',X'^\vee,R'^\vee,\Delta'^\vee)$ is another based root datum with associated dominant weights $(X'_+,\leq')$, then any ismorphism $(X,\leq)\to(X',\leq')$ of partially ordered semigroups prolongs to an isomorphism $\calB\to\calB'$ of based root data.
\end{lem}       
\begin{proof}
The weight lattice $X$ is the group completion of $X_+$, a finite free $\bbZ$-module. The dominance order $\leq$ extends uniquely to $X$, also denoted $\leq$. Then $X^\vee=\Hom_\bbZ(X,\bbZ)$ is the coweight lattice, and the natural pairing $X\times X^\vee\to\bbZ$ identifies with $\langle\str,\str\rangle$. The reconstruction of the roots and coroots proceeds in several steps: \smallskip \\
(1) {\it The set of simple roots} $\Delta\subset X$: \\
A weight $\al\in X\backslash\{0\}$ is in $\Delta$ if and only if $0\leq\al$, and $\al$ is minimal with this property. \smallskip \\
(2) {\it The set of simple coroots} $\Delta^\vee\subset X^\vee$: \\
An element of $X^\vee$ is uniquely determined by its value on $X_+$. Fix $\al\in\Delta$ with corresponding simple coroot $\al^\vee$. Then for any $\mu\in X_+$, the value $\langle\mu,\al^\vee\rangle$ is the unique number $m\in\bbN$ such that $2\mu-m\al$ is dominant, but $2\mu-(m+1)\al$ is not. Indeed, we have 
\[\langle2\mu-m\al,\al^\vee\rangle\geq 0 \;\;\; \Leftrightarrow \;\;\; \langle\mu,\al^\vee\rangle\geq m,\]  
and, for every other simple coroot $\beta^\vee\not=\al^\vee$ and every $n\in\bbN$,
\[\langle2\mu-n\al,\beta^\vee\rangle=2\langle\mu,\beta^\vee\rangle-n\langle\al,\beta^\vee\rangle\geq2\langle\mu,\beta^\vee\rangle\geq 0,\]
since $\langle\al,\beta^\vee\rangle\leq 0$. Hence, $\langle2\mu-(m+1)\al,\al^\vee\rangle< 0$ and so $m=\langle\mu,\al^\vee\rangle$. \smallskip \\
(3) {\it The sets of roots $R$ and coroots} $R^\vee$:\\
The Weyl group $W\subset\Aut_\bbZ(X)$ is the finite subgroup generated by the reflections $s_{\al,\al^\vee}$ associated to the pair $(\al,\al^\vee)\in\Delta\times\Delta^\vee$. Then $R=W\cdot\Delta$, i.e., the roots are given by the translates of the simple roots under $W$. Since $\Aut_\bbZ(X^\vee)=\Aut_\bbZ(X)^\op$, the Weyl group $W$ acts on $X^\vee$ and $R^\vee=W\cdot\Delta^\vee$. This proves the lemma.
\end{proof}                                  
                    
\end{appendix}


\begin{thebibliography}{99}

\bibitem{BL} 
A. Beauville and Y. Laszlo: {\it Un lemme de descente}, C. R. Acad. Sci. Paris S\'er. I Math. 320 (1995), no. 3, 335-340. 




\bibitem{BrKu} 
M. Brion and S. Kumar: {\it Frobenius Splitting Methods in Geometry and Representation Theory}, Birkh\"auser Boston Inc., Boston, MA (2005).

\bibitem{DM} 
P. Deligne and J. Milne: {\it Tannakian categories}, in Hodge Cycles and Motives, Lecture Notes in Math. 900 (1982) 101¨C228.

\bibitem{Fa} 
G. Faltings: {\it Algebraic loop groups and moduli spaces of bundles}, J. Eur. Math. Soc. (JEMS) 5 (2003), no. 1, 41Ð68.

\bibitem{Ga} 
D. Gaitsgory: {\it Construction of central elements in the affine Hecke algebra via nearby cycles}, Invent. Math. 144 (2001), no. 2, 253--280.

\bibitem{GHKR} 
U. G\"ortz, T. Haines, R. Kottwitz and D. Reuman: {\it Dimensions of some affine Deligne-Lusztig varieties}, Ann. Sci. ƒcole Norm. Sup. (4) 39 (2006), no. 3, 467Ð511.

\bibitem{Haines} 
T. Haines: {\it Structure constants for Hecke and representation rings}, Int. Math. Res. Not. 2003, no. 39, 2103Ð2119.




\bibitem{KLV} 
D. Kazhdan, M. Larsen and Y. Varshavsky: {\it The Tannakian formalism and the Langlands conjectures}, preprint 2010, arXiv:1006.3864.

\bibitem{Kumar} 
S. Kumar: {\it Proof of the Parthasarathy-Ranga Rao-Varadarajan conjecture}, Invent. Math. 102 (1990), no. 2, 377-398.

\bibitem{Lu} 
G. Lusztig: {\it Singularities, character formulas, and a $q$-analogue of weight multiplicities}, Analysis and Topology on Singular Spaces, II, III, Ast\'erisque, Luminy, 1981, vol. 101Ð102, Soc. Math. France, Paris (1983), pp. 208-229.

\bibitem{LO} 
Y. Laszlo and M. Olsson: {\it Perverse t-structure on Artin stacks}, Math. Z. 261 (2009), no. 4, 737-748.

\bibitem{LS} 
Y. Laszlo and C. Sorger: {\it The line bundles on the moduli of parabolic $G$-bundles over curves and their sections}, Ann. Sci. \'Ecole Norm. Sup. (4) 30 (1997), no. 4, 499-525.

\bibitem{Mathieu} 
O. Mathieu: {\it Construction dÕun groupe de Kac-Moody et applications}, Compositio Math. 69 (1989), 37-60.

\bibitem{MV} 
I. Mirkovi\'c and K. Vilonen: {\it Geometric Langlands duality and representations of algebraic groups over commutative rings},  Ann. of Math. (2)  166  (2007),  no. 1, 95-143.

\bibitem{NP} 
Ng\^o B. C. and P. Polo: {\it R\'esolutions de Demazure affines et formule de Casselman-Shalika g\'eom\'etrique}, J. Algebraic Geom. 10 (2001), no. 3, 515-547.

\bibitem{Ra} 
M. Rapoport: {\it A guide to the reduction modulo $p$ of Shimura varieties}, Ast\'erisque No. 298 (2005), 271-318.

\bibitem{RR} 
R. Reich: {\it Twisted geometric Satake equivalence via gerbes on the factorizable grassmannian}, preprint 2010, arXiv:1012.5782v3.

\bibitem{RZ} 
X. Zhu: {\it The Geometrical Satake Correspondence for Ramified Groups} with an appendix by T. Richarz and X. Zhu, preprint 2011, arXiv:1107.5762v1.



\end{thebibliography}
\end{document}